\documentclass[12pt]{amsart}
\usepackage{amsmath,amssymb,amsbsy,amsfonts,amsthm,latexsym,
                        amsopn,amstext,amsxtra,euscript,amscd,color}
\usepackage{graphicx}

\usepackage{float} 
\begin{document}
\bibliographystyle{plain}

\newfont{\teneufm}{eufm10}
\newfont{\seveneufm}{eufm7}
\newfont{\fiveeufm}{eufm5}
%
%
\newfam\eufmfam
     \textfont\eufmfam=\teneufm \scriptfont\eufmfam=\seveneufm
     \scriptscriptfont\eufmfam=\fiveeufm
%
%
\def\frak#1{{\fam\eufmfam\relax#1}}
%


\def\bbbr{{\rm I\!R}} 
\def\bbbr{\mathbb{R}}
\def\bbbm{{\rm I\!M}}
\def\bbbm{\mathbb{M}}
\def\bbbn{{\rm I\!N}} 
\def\bbbn{\mathbb{N}}
\def\bbbf{{\rm I\!F}}
\def\bbbf{\mathbb{F}}
\def\bbbh{{\rm I\!H}}
\def\bbbh{\mathbb{H}}
\def\bbbk{{\rm I\!K}}
\def\bbbk{\mathbb{K}}
\def\bbbp{{\rm I\!P}}
\def\bbbp{\mathbb{P}}

\def\bbbone{{\mathchoice {\rm 1\mskip-4mu l} {\rm 1\mskip-4mu l} {\rm
1\mskip-4.5mu l}  {\rm 1\mskip-5mu l}}}
\def\bbbc{{\mathchoice {\setbox0=\hbox{$\displaystyle\rm C$}
\hbox{\hbox to0pt{\kern0.4\wd0\vrule height0.9\ht0\hss}\box0}}
{\setbox0=\hbox{$\textstyle\rm C$}
\hbox{\hbox to0pt{\kern0.4\wd0\vrule height0.9\ht0\hss}\box0}}
{\setbox0=\hbox{$\scriptstyle\rm C$}\hbox{\hbox to0pt{\kern0.4\wd0\vrule
height0.9\ht0\hss}\box0}} {\setbox0=\hbox{$\scriptscriptstyle\rm C$}
\hbox{\hbox to0pt{\kern0.4\wd0\vrule height0.9\ht0\hss}\box0}}}}

\def\bbbq{{\mathchoice {\setbox0=\hbox{$\displaystyle\rm Q$}
\hbox{\raise 0.15\ht0\hbox to0pt{\kern0.4\wd0\vrule height0.8\ht0\hss}\box0}}
{\setbox0=\hbox{$\textstyle\rm Q$}
\hbox{\raise 0.15\ht0\hbox to0pt{\kern0.4\wd0\vrule height0.8\ht0\hss}\box0}}
{\setbox0=\hbox{$\scriptstyle\rm Q$}
\hbox{\raise 0.15\ht0\hbox to0pt{\kern0.4\wd0\vrule height0.7\ht0\hss}\box0}}
{\setbox0=\hbox{$\scriptscriptstyle\rm Q$}
\hbox{\raise 0.15\ht0\hbox to0pt{\kern0.4\wd0\vrule height0.7\ht0\hss}\box0}}}}
\def\bbbq{\mathbb{Q}}

\def\bbbt{{\mathchoice  {\setbox0=\hbox{$\displaystyle\rm T$}\hbox{\hbox
to0pt{\kern0.3\wd0\vrule height0.9\ht0\hss}\box0}}
{\setbox0=\hbox{$\textstyle\rm T$}
\hbox{\hbox to0pt{\kern0.3\wd0\vrule height0.9\ht0\hss}\box0}}
{\setbox0=\hbox{$\scriptstyle\rm T$}
\hbox{\hbox to0pt{\kern0.3\wd0\vrule height0.9\ht0\hss}\box0}}
{\setbox0=\hbox{$\scriptscriptstyle\rm T$} hbox{\hbox to0pt{\kern0.3\wd0\vrule
height0.9\ht0\hss}\box0}}}}
\def\bbbt{\mathbb{T}}

\def\bbbs{{\mathchoice
{\setbox0=\hbox{$\displaystyle     \rm S$}\hbox{\raise0.5\ht0\hbox
to0pt{\kern0.35\wd0\vrule height0.45\ht0\hss}\hbox to0pt{\kern0.55\wd0\vrule
height0.5\ht0\hss}\box0}} {\setbox0=\hbox{$\textstyle        \rm
S$}\hbox{\raise0.5\ht0\hbox to0pt{\kern0.35\wd0\vrule height0.45\ht0\hss}\hbox
to0pt{\kern0.55\wd0\vrule height0.5\ht0\hss}\box0}}
{\setbox0=\hbox{$\scriptstyle      \rm S$}\hbox{\raise0.5\ht0\hbox
to0pt{\kern0.35\wd0\vrule height0.45\ht0\hss}\raise0.05\ht0\hbox
to0pt{\kern0.5\wd0\vrule height0.45\ht0\hss}\box0}}
{\setbox0=\hbox{$\scriptscriptstyle\rm S$}\hbox{\raise0.5\ht0\hbox
to0pt{\kern0.4\wd0\vrule height0.45\ht0\hss}\raise0.05\ht0\hbox
to0pt{\kern0.55\wd0\vrule height0.45\ht0\hss}\box0}}}}
\def\bbbs{\mathbb{S}}

\def\bbbz{{\mathchoice {\hbox{$\sf\textstyle Z\kern-0.4em Z$}}
{\hbox{$\sf\textstyle Z\kern-0.4em Z$}} {\hbox{$\sf\scriptstyle Z\kern-0.3em
Z$}} {\hbox{$\sf\scriptscriptstyle Z\kern-0.2em Z$}}}} \def\ts{\thinspace}
\def\bbbz{\mathbb{Z}}

\newtheorem{theorem}{Theorem} \newtheorem{lemma}[theorem]{Lemma}
\newtheorem{claim}[theorem]{Claim}
\newtheorem{cor}[theorem]{Corollary}
\newtheorem{prop}[theorem]{Proposition}
\newtheorem{definition}{Definition}
\newtheorem{question}[theorem]{Open Question}

\def\qed{\ifmmode
\squareforqed\else{\unskip\nobreak\hfil
\penalty50\hskip1em\null\nobreak\hfil\squareforqed
\parfillskip=0pt\finalhyphendemerits=0\endgraf}\fi}

\def\squareforqed{\hbox{\rlap{$\sqcap$}$\sqcup$}}

\def\cA{{\mathcal A}}
\def\cB{{\mathcal B}}
\def\cC{{\mathcal C}}
\def\cD{{\mathcal D}}
\def\cE{{\mathcal E}}
\def\cF{{\mathcal F}}
\def\cG{{\mathcal G}}
\def\cH{{\mathcal H}}
\def\cI{{\mathcal I}}
\def\cJ{{\mathcal J}}
\def\cK{{\mathcal K}}
\def\cL{{\mathcal L}}
\def\cM{{\mathcal M}}
\def\cN{{\mathcal N}}
\def\cO{{\mathcal O}}
\def\cP{{\mathcal P}}
\def\cQ{{\mathcal Q}}
\def\cR{{\mathcal R}}
\def\cS{{\mathcal S}}
\def\cT{{\mathcal T}}
\def\cU{{\mathcal U}}
\def\cV{{\mathcal V}}
\def\cW{{\mathcal W}}
\def\cX{{\mathcal X}}
\def\cY{{\mathcal Y}}
\def\cZ{{\mathcal Z}}
\newcommand{\rmod}[1]{\: \mbox{mod}\: #1}

\def\Tr{{\mathrm{Tr}}}

\def\mand{\qquad \mbox{and} \qquad} \renewcommand{\vec}[1]{\mathbf{#1}}

\def\HNP{{\mathcal{HNP_{\mu}}}} 

\def\tf{\widetilde f}

\def\eqref#1{(\ref{#1})}


\newcommand{\ignore}[1]{}

\hyphenation{re-pub-lished}

\def\lln{{\mathrm Lnln}}
\def\ad{{\mathrm ad}}

\def\F{{\bbbf}}
\def\K{{\bbbk}}
\def \Z{{\bbbz}}
\def \N{{\bbbn}}
\def\Q{{\bbbq}}
\def \R{{\bbbr}}
\def\Fp{\F_p}
\def\fp{\Fp^*}
\def\Zm{\Z_m}
\def \Um{{\mathcal U}_m}

\def \Bf{\frak B}

\def\Km{\cK_\mu}

\def\va {{\mathbf a}}
\def \vb {{\mathbf b}}
\def \vc {{\mathbf c}}
\def\vx{{\mathbf x}}
\def \vr {{\mathbf r}}
\def \vs {{\mathbf s}}
\def \vv {{\mathbf v}}
\def\vu{{\mathbf u}}
\def \vw{{\mathbf w}}
\def \vz {{\mathbf z}}
\def \ve {{\mathbf e}}
\def \vp {{\mathbf p}}

\def \vol {\mathrm{Vol}}

\def\dist {{\mathrm{\,dist\,}}}

\def\\{\cr}
\def\({\left(}
\def\){\right)}
\def\fl#1{\left\lfloor#1\right\rfloor}
\def\rf#1{\left\lceil#1\right\rceil}

\def\near#1{\left\lfloor #1\right\rceil}
\def\cvp#1{\left\lfloor{\vec{#1}}\right\rceil}

\def\flp#1{{\left\langle#1\right\rangle}_p}
\def\flm#1{{\left\langle#1\right\rangle}_m}

\def\dmod#1#2{\left|#1\right|_{#2}}

\def\Al{{\sl Alice}}
\def\Bob{{\sl Bob}}

\def\Or{{\mathcal O}}

\newcommand{\commR}[1]{\marginpar{%
\vskip-\baselineskip 
\raggedright\footnotesize
\itshape\hrule\smallskip\begin{color}{blue}#1\end{color}\par\smallskip\hrule}}

\newcommand{\commI}[1]{\marginpar{%
\vskip-\baselineskip 
\raggedright\footnotesize
\itshape\hrule\smallskip\begin{color}{red}#1\end{color}\par\smallskip\hrule}}

\def\inv#1{\mbox{\rm{inv}}\,#1} \def\invM#1{\mbox{\rm{inv}}_M\,#1}
\def\invp#1{\mbox{\rm{inv}}_p\,#1}

\def\Ln#1{\mbox{\rm{Ln}}\,#1}

\def \nd {\,|\hspace{-1.2mm}/\,}

\def\ord{\mu}

\def\e {\mbox{\bf{e}}}

\def\epp{\mbox{\bf{e}}_{p-1}}
\def\ep{\mbox{\bf{e}}_p}
\def\eq{\mbox{\bf{e}}_q}

\newcommand{\floor}[1]{\lfloor {#1} \rfloor}

\def\rem{{\mathrm{\,rem\,}}}
\def\distp {{\mathrm{\,dist_p\,}}}
\def\etal{{\it et al.}}
\def\ie{{\it i.e. }}
\def\veps{{\varepsilon}}
\def\eps{{\eta}}

\def\ind#1{{\mathrm {ind}}\,#1}
      \def \MSB{{\mathrm{MSB}}}
\newcommand{\abs}[1]{\left| #1 \right|}

\title[Noisy polynomial interpolation and approximation]{
Interpolation and Approximation 
of Polynomials in Finite Fields
over a Short Interval
from Noisy  Values}

\author[Garcia-Morchon]{Oscar Garcia-Morchon}
\address{Philips Research Laboratories, High Tech Campus 34,
5656 AE Eindhoven, The Netherlands}
\email{oscar.garcia@philips.com}

\author[Rietman]{Ronald Rietman}
\address{Philips Research Laboratories, High Tech Campus 34,
5656 AE Eindhoven, The Netherlands}
\email{ronald.rietman@philips.com}

\author[Shparlinski] {Igor E. Shparlinski} 
\address{Department of Pure Mathematics, University of New South Wales, 
Sydney, NSW 2052, Australia}
\email{igor.shparlinski@unsw.edu.au}

\author[Tolhuizen]{Ludo Tolhuizen}
\address{Philips Research Laboratories, High Tech Campus 34,
5656 AE Eindhoven, The Netherlands}
\email{ludo.tolhuizen@philips.com}

\date{\today}

\begin{abstract}
Motivated by a recently introduced HIMMO key 
distribution scheme, we consider a modification 
of the {\it noisy polynomial interpolation problem\/}
of recovering an unknown polynomial $f(X) \in \Z[X]$ 
from approximate values of the residues of $f(t)$ modulo a 
prime $p$ at polynomially many points $t$ taken from 
a short interval. 
\end{abstract}

\keywords{
Noisy polynomial interpolation,
 finite fields, lattice reduction, HIMMO key distribution scheme}

\maketitle

\section{Introduction}

\subsection{Motivation}

Here,  we consider the following 
problem: 
\begin{quotation}
Given a prime $p$, recover
an unknown polynomial $f$ over a finite field $\F_p$ of $p$ elements 
from several approximations to the values $f(t)$, computed at several 
points $t$. 
\end{quotation}
 
Several problems of this type are related to the so-called
{\it hidden number problem\/} introduced by Boneh and 
Venkatesan~\cite{BV1,BV2}, and have already been 
studied intensively due to their cryptographic relevance, see
the survey~\cite{Shp2}. 

Usually the  evaluation points $t$ are chosen from the whole field $\F_p$.
However,  here, motivated  by the links and possible applications to the
recently introduced HIMMO key distribution scheme~\cite{GRTGGM},
we concentrate on a very different case, which has never been discussed 
in the literature. Namely, in the settings relevant to HIMMO, 
the values of $t$ are taken from a short interval 
rather than from the whole field $\F_p$.
This case requires a careful adaptation  of existing algorithms
of~\cite{Shp1,ShpWint} 
and also establishing new number theoretic results about 
the frequency of small residues of polynomial values 
evaluated at a small argument, which are based on some 
ideas from~\cite{CCGHSZ,CGOS}.

We remark the polynomial recovery problem as studied in this paper arises in 
a collusion attack on a single node in the HIMMO system. 
Collusion attacks on the complete HIMMO system lead to recovery problems 
involving several polynomials reduced modulo several distinct unknown integers, 
which we believe to be much more difficult.


\subsection{Polynomial noisy interpolation and approximation problems}

For a prime $p$ we denote by $\F_p$
the field of $p$ elements. We identify the elements of $\F_p$ with the set
$\{0,\ldots, p-1\}$, so we can talk about their bits 
(for example most and least significant bits) and also 
approximations.

We consider the {\it noisy polynomial  interpolation problem\/} of
finding an unknown polynomial
$$
f(X) = \sum_{j=0}^n a_j X^{j} \in \F_p[X],
$$
of degree $n$ from approximations to values of  $f(t)$ (treated as  elements
of the set $\{0,\ldots, p-1\}$) 
at polynomially many points
$t \in\F_p$
selected uniformly at random.

More precisely, 
for integers $s$ and $m \ge 1$ we denote by  $\dmod{s}{m}$
the distance between $s$ and the closest multiple of $m$, that is,
$$
\dmod{s}{m} = \min_{k \in \Z} |s - km|.
$$
Then, these approximations can, for example,  be given as 
integers $u_t$ with 
\begin{equation}
\label{eq:Approx}
\dmod{u_t - f(t)}{p} \le \Delta
\end{equation}
for some ``precision'' $\Delta$ for values of $t$ that are 
chosen uniformly at random from a certain set $\cT$ of 
``samples''. 

In cryptographic applications, 
the approximations are usually given by strings of $s$ most or 
least significant bits of $f(t)$, where, 
if necessary, some leading zeros are added to make sure that $f(t)$ 
is represented by $r$ bits, where $r$ is the bit length of $p$. 

It is clear that giving the most significant bits is essentially equivalent 
to giving approximations of the type~\eqref{eq:Approx} with an appropriate $\Delta$.
We now observe that giving the least significant bits can be recast 
as a question of this type as well. Indeed, assume we are given some $s$-bit 
integer $v_t$ 
so that 
$$
f(t) = 2^{s}w_t + v_t
$$
for some (unknown) integer $w_t$. We now have $0 \le w_t < 2^{r-s}$
so, if $\lambda$ is the multiplicative inverse of $2^{s}$ in 
$\F_p$ then, setting $u_t = \lambda v_t + 2^{r-s-1}$, we obtain an inequality 
of the type~\eqref{eq:Approx} with the polynomial $\lambda f(X)$ 
instead of $f(X)$ and $\Delta = 2^{r-s-1}$. 

So from now on we concentrate on the case where the 
approximations to $f(t)$ are given in the
form of~\eqref{eq:Approx}. 


The case of linear polynomials
corresponds to the hidden number problem introduced by Boneh and 
Venkatesan~\cite{BV1,BV2}. The case of general polynomials, including 
sparse polynomials of very large degree, has been studied in~\cite{Shp1,ShpWint}
(see also~\cite{Shp2} for a survey of several other problems of 
similar types). We note that if the polynomial $f$ belongs to some 
family $\cF$ of polynomials (such as polynomials of degree $d$ or sparse 
polynomials with $d$ monomials), it is crucial for the algorithms
of~\cite{Shp1,ShpWint} to have some uniformity results for every
non-zero polynomial of $\cF$ on the values of $t$ 
chosen uniformly at random from $\cT$. We also note that a multiplicative 
analogue of this problem has been
studied in~\cite{vzGShp}.

We recall that in the settings of all previous works, typically the set $\cT$ 
is the whole field $\F_p$. In turn, 
this enables the use of such a powerful number theoretic
technique as the method of exponential sums and thus the use
of the Weil bound (see~\cite[Theorem~5.38]{LN}) for dense polynomials 
and the bound 
of Cochrane, Pinner and Rosenhouse~\cite{CPR} for sparse polynomials. 
In particular, in~\cite{ShpWint} a polynomial time algorithm is designed that 
works with very large values of  $\Delta$, namely, 
up to $\Delta = p  \exp\(- c(w
\log p)^{1/2}\)$, provided that $f$ is a sparse polynomial of degree 
$n \le  p^{1/2} \exp\(- c (w \log p)^{1/2}\)$ and with $w$ monomials
(of known degrees), where $c>0$ is some absolute constant. 
The analysis of the algorithm of~\cite{ShpWint}  is
based on the method of exponential sums,  the bound of~\cite{CPR}
and some ideas related to the Waring problem.

We also note that if $\cT$ is an interval of length at least 
$p^{1/n + \varepsilon}$ for some fixed $\varepsilon > 0$, where,
as before $n = \deg f$, then one can use the 
method of~\cite{Shp1,ShpWint} augmented with the bounds of exponential sums
obtained via 
the method of Vinogradov, see the striking results of 
Wooley~\cite{Wool1,Wool2,Wool3}.
However, here we are interested in  shorter intervals of the 
form 
$$
\cT = [-h,h],
$$ 
which appear naturally in the construction of~\cite{GRTGGM}. 
Thus instead of the method of exponential sums, we use some ideas
from~\cite{CCGHSZ,CGOS}.

We remark that for the noisy interpolation to succeed,  
we always assume that some low order the coefficients 
$a_0, \ldots, a_{k-1}$ of $f$ are known, where $k$ is such that 
$h^k > \Delta^{1+\varepsilon}$ for some fixed $\varepsilon > 0$.  
Clearly such a condition is necessary as if $h^k \le  \Delta$ 
then the approximations
of the type~\eqref{eq:Approx} are not likely to 
be sufficient to distinguish between 
$$
f_1(X) = \sum_{j=k}^n a_j X^{j} \mand f_2(X) = \sum_{j=k}^n a_j X^{j} + X^{k-1}
$$
from approximations~\eqref{eq:Approx}  at random $t \in [-h,h]$. 
Furthermore, it is clear that without loss of generality we can always
assume that 
$$
a_0 = \ldots = a_{k-1} = 0.
$$
Note that in the case when the test set $\cT$ is the whole field $\F_p$ 
we only need to request that the constant coefficient $a_0$ of $f$ is known 
(which is also, always assumed to be zero, see~\cite{BV1,BV2,Shp1,ShpWint}). 

The above example, which shows the limits of 
{\it interpolation\/} of $f$ from the information 
given by~\eqref{eq:Approx}, 
motivates the following question of {\it approximation\/} to $f$, which is 
also more relevant to attacking HIMMO. Namely, instead of 
finding a polynomial $f$ we ask whether we can find a polynomial $\tf$,
such that $f(t)$ and $\tf(t)$ are close to each other for 
all $t \in [-h,h]$ so the approximations~\eqref{eq:Approx}  do not 
allow to distinguish between $f$ and $\tf$. 
In the terminology of~\cite{GRTGGM} both $f$ and $\tf$ lead to the same
keys and thus the attacker can use $\tf$ instead of $f$.

Finally, we note that the algorithmic problems considered in this 
paper have led us to some new problems which are of intrinsic
number theoretic interest. 

\subsection{Approach and structure}

Generally, our approach follows that of~\cite{Shp1,ShpWint} which 
in turn is based on the idea of Boneh and 
Venkatesan~\cite{BV1,BV2}.
Thus lattice algorithms, namely the algorithm 
for the closest vector problem, see Section~\ref{sec:Lat CVP}, and 
a link between this problem and polynomial approximations, 
see  Section~\ref{sec:Lap Poly}, play a crucial role in 
our approach. 

However, the analysis of our algorithm requires very 
different tools compared with those used in~\cite{Shp1,ShpWint}.
Namely we need to establish some results about the frequency 
of small polynomial values at small arguments, which we derive in
Section~\ref{sec:PolyVal}, 
closely following some ideas from~\cite{CCGHSZ,CGOS}. 

It also clear that there is a natural limit for such estimates on 
the frequency of small values as any polynomial with 
small coefficients takes small values at small arguments 
(it is certainly easy to quantify the notion of ``small''
in this statement). In fact our bound of Lemma~\ref{lem:NFIJ-1}
is nontrivial up to exactly this limit. 

Rather unexpectently,  in Section~\ref{sec:Except Poly}
we show that there are several other types of polynomials 
which satisfy this property: that is, they take small values at 
small arguments even if the coefficients are quite large
(certainly this may only happen beyond the range of the bound of 
Lemma~\ref{lem:NFIJ-1} which applies to all polynomials). 

We call such polynomials {\it 
exceptional\/}.  Then, after recalling in Section~\ref{sec:Prep} 
some result from analytic number theory, in Section~\ref{sec:Coeff Except} 
we give some results describing 
the structure of the coefficients of exceptional polynomials. 
Note the presence of exceptional polynomials makes noisy polynomial 
interpolation
impossible for a wide class of instances (which is actually good news for 
the security of HIMMO).  So it is important to understand the structure 
and the frequency of exceptional polynomials. The results of Section~\ref{sec:Coeff Except} 
provide some partial progress toward this goal, yet many important 
questions are still widely open. 

Our main results are presented in Sections~\ref{sec:Poly Int} and~\ref{sec:Poly App}.
In Section~\ref{sec:Poly Int} we treat the noisy polynomial 
interpolation problem, where we actually want to recover the hidden 
polynomial $f$. However, from the cryptographic point of view it is 
enough to recover another polynomial $g$ which for all or most of small 
arguments takes values close to those taken by $f$. 
Here the existence 
of exceptional polynomials becomes of primal importance. 
We use the results of Section~\ref{sec:Coeff Except}  to 
obtain an upper bound on the number of such polynomials $g$. 

Furthermore, in Sections~\ref{sec:Flat Poly}
and~\ref{sec:Oscill Poly}
 we also give two interesting explicit constructions of such 
exceptional polynomials, which we illustrate by some 
concrete numerical examples. 

In Section~\ref{sec:Poly App} we treat the approximate recovery problem and
give a formula that predicts whether an approximate recovery is likely to be
successful, depending on the number of randomly chosen
observation points. We compare its
predictions with numerical experiments for parameter values suitable for
HIMMO.

\subsection{Notation}

Throughout the paper, the implied constants in the symbols `$O$',  `$\ll$' 
and `$\gg$'
may occasionally, where obvious, depend on  the degrees of 
the polynomials involved and on the real parameter $\varepsilon$ 
and are absolute, otherwise.
We recall that the notations $U = O(V)$,  $U \ll V$  and $V \gg U$ are all
equivalent to the assertion that the inequality $|U|\le c|V|$ holds for some
constant $c>0$.

The letter $p$ always denotes a prime number, while the letters $h$, $k$,
$\ell$, $m$ 
and $n$ (in lower and upper cases) always denote positive integer numbers.
%

\section{Distribution of Values of Polynomials}

\subsection{Polynomial values in a given  interval} 
\label{sec:PolyVal} 

For a polynomial 
$$
F(X) \in \F_p[X]
$$
and two intervals 
$$\cI = \{u+1, \ldots, u+H\}\mand \cJ = \{v+1, \ldots, v+K\}
$$ 
of $H$ and $K$ consecutive integers, respectively, 
we denote by $N_{F}(\cI, \cJ)$ the number of 
values $t \in \cI$ for which $F(t) \in \cJ$ 
(where the elements of the intervals  $\cI$ and $\cJ$ are embedded into
$\F_p$ under the natural reduction modulo $p$).

For two intervals $\cI$ 
and $\cJ$ of the same length, various bounds on $N_{F}(\cI, \cJ)$ are 
given in~\cite{CCGHSZ,CGOS}. 
It is easy to see that the argument of the proof of~\cite[Theorem~1]{CGOS} allows 
us to  estimate $N_{F}(\cI, \cJ)$ for intervals of $\cI$ and $\cJ$ 
of different lengths as well. 

To present this result, 
for positive integers $k$, $\ell$ and $H$, we denote by
$J_{k,\ell}(H)$ the number of solutions to the system of
equations
$$
x_1^{\nu}+\ldots +x_k^{\nu}=x_{k+1}^{\nu}+\ldots +x_{2k}^{\nu},
\qquad \nu = 1, \ldots, \ell, 
$$
with
$$
1 \le x_1, \ldots, x_{2k} \le H.
$$
Next, we define $\kappa(\ell)$ as the smallest integer $\kappa$ such
that for $k \ge \kappa$ there exists a constant
$C(k,\ell)$ depending only on $k$ and $\ell$ and such that 
$$
J_{k,\ell}(H) \le C(k,\ell) H^{2k - \ell(\ell+1)/2+o(1)}
$$
holds as   $H\to \infty$. Presently, the strongest available upper bound
$$
\kappa(\ell) \le \ell^2-\ell + 1
$$ 
is due to Wooley~\cite{Wool3}, see also~\cite[Theorem~1.1]{Wool2} and~\cite[Theorem~1.1]{Wool1}; 
note that, in particular,
these bounds also imply the existence of $\kappa(\ell)$ which is 
a very nontrivial fact.

It is now  easy to see that the proof of~\cite[Theorem~1]{CGOS} can be
generalised to  lead to the following bound.

\begin{lemma}
\label{lem:NFIJ-1} Let $F  \in \F_p[X]$ be a polynomial of degree $\ell$. 
Then  for 
intervals $\cI = \{u+1, \ldots, u+H\}$ 
and $\cJ = \{v+1, \ldots, v+K\}$ with $1 \le H, K < p$ we have
$$
N_F(\cI,\cJ) \le  H^{1+o(1)} \((K/p)^{1/2\kappa(\ell)} + (K/H^\ell)^{1/2\kappa(\ell)}\),
$$
as $H\to\infty$. 
\end{lemma}

Note that Kerr~\cite[Theorem~3.1]{Ker} gives a more general 
form of Lemma~\ref{lem:NFIJ-1} that applies to multivariate 
polynomials and also to congruences modulo a composite 
number. 

We also note, that  several other results from~\cite{CCGHSZ,CGOS} can 
be extended from the case $H=K$ to the general case as well, but this does 
not affect our main result. 

Clearly the bound of Lemma~\ref{lem:NFIJ-1} is nontrivial 
provided that $K \le p^{1-\varepsilon}$ and also $H^\ell > K^{1 + \varepsilon}$
for a fixed $\varepsilon > 0$ and a sufficiently large $H$. 
It is easy to see that this conditions cannot be substantially improved as
it is clear that if $K = p$ then $N_F(\cI,\cJ) = H$ for any polynomial 
$F  \in \F_p[X]$. Also, if $K = H^\ell$ then $N_F(\cI,\cJ) = H$
for $\cI = \{1, \ldots, H\}$, $\cJ = \{1, \ldots , K\}$ 
and $F(X) = X^\ell$. 
The above example may suggest that unless the coefficients of $F$ are 
small, the value of $N_F(\cI,\cJ)$ can be nontrivially estimated
even for $K > H^\ell$. We now show that this is false
and in fact there are many other polynomials, which we call {\it
exceptional\/}, which have the same property.

\subsection{Exceptional polynomials}
\label{sec:Except Poly}

Let us fix integers $s\ge 2$, $0\le r_i < s^i$, 
$i=0, \ldots, \ell$ and 
 $K > \ell H^\ell$.
 We now 
consider a polynomial 
$$
F(X) = \sum_{i=0}^\ell A_i X^i \in \F_p[X],
$$
where 
\begin{equation}
\label{eq:exam}
A_i \in  \left[\frac{r_i}{s^i}p, \frac{r_i}{s^i}p + \frac{K}{\ell H^{i}}\right], \qquad 
i = 0, \ldots, \ell.
\end{equation}
In particular
$A_is^i \equiv B_i \pmod p$ for some integers 
$$
B_i \in \left[0,  \frac{K s^i}{\ell H^{i}}\right], \qquad 
i = 0, \ldots, \ell.
$$
Then for $t = su$ with an integer $u \in [1, H/s]$ we
have 
$$
F(t) = F(su) = \sum_{i=0}^\ell A_i (su)^i \equiv
\sum_{i=0}^\ell B_i u^i \pmod p.
$$
Hence $F(t) \in [0, K]$ for such values of $t$, which implies the inequality 
$N_F(\cI,\cJ)  \ge H/s -1$ for the above example. Taking 
$s$ to be small, we see that $F$ has quite large 
coefficients, however no nontrivial bound on 
$N_F(\cI,\cJ)$ is possible for  $K > \ell H^\ell$.

In fact, in Sections~\ref{sec:Flat Poly} and~\ref{sec:Oscill Poly}
we  give examples of polynomials 
with large coefficients that remain small 
for all  small values of the arguments, that is,
polynomials for which $N_F(\cI,\cJ) = H$, 
for some rather short intervals $\cJ$ compared 
with the length of $\cI$.

In the above examples, the ratios of the coefficients 
and $p$ are chosen to be very close to rational numbers with small numerators
and denominators. 
Alternatively, 
this can be written as the condition that the coefficients 
of $F$ satisfy
$$
A_i v \equiv u_i \pmod p
$$
for some small integers $v, u_i$, $i =0, \ldots, \ell$.

Below we present several results that demonstrate that
this property indeed captures the class of exceptional polynomials $F$
with large values of $N_F(\cI,\cJ)$ adequately, at least 
in the qualitative sense.  

First we need to recall some number theoretic results. 

\subsection {Some preparations}
\label{sec:Prep}

The following result is well-known and can be found, for example, in~\cite[Chapter~1, Theorem~1]{Mont}
(which is a  more precise form of the celebrated Erd{\H o}s--Tur\'{a}n inequality).

\begin{lemma}
\label{lem:ET small int}
Let $\gamma_1, \ldots, \gamma_H$ be a sequence of $H$ points of the unit interval $[0,1]$.
Then for any integer $R\ge 1$, and an interval $[\alpha, \beta] \subseteq [0,1]$,
we have
\begin{equation*}
\begin{split}
\# \{t=1, \ldots, H~:&~\gamma_t  \in [\alpha, \beta]\} - H(\beta - \alpha)\\
\ll \frac{H}{R} + &\sum_{r=1}^R \(\frac{1}{R} +\min\{\beta - \alpha, 1/r\}\)
\left|\sum_{t=1}^H \exp(2 \pi i r \gamma_t)\right|.
\end{split}
\end{equation*}
\end{lemma}

To use Lemma~\ref{lem:ET small int} we also need an estimate on exponential sums
with polynomials,  which is essentially due to Weyl, see~\cite[Proposition~8.2]{IwKow}.

Let $\dmod{\xi}{1} = \min\{|\xi - k|~:~k\in \Z\}$ denote the distance
between a real $\xi$ and the closest integer (which can be considered as a
modulo 1 version of the notation $\dmod{z}{p}$). 

\begin{lemma}
\label{lem:Weyl}
Let $\psi(X) \in \R[X]$ be a polynomial of degree $\ell\ge 2$ with the leading
coefficient $\vartheta \ne 0$.
Then
\begin{equation*}
\begin{split}
&\left|\sum_{t=1}^H \exp(2 \pi i \psi(t))\right|\\
&\qquad \ll H^{1-\ell/2^{\ell-1}}  \(\sum_{-H < t_1 , \ldots,  t_{\ell-1}  < H}
\min\left\{H, \frac{1}{\dmod{\vartheta \ell!  t_1  \cdots
t_{\ell-1}}{1}}\right\}\)^{1/2^{\ell-1}}.
\end{split}
\end{equation*}
\end{lemma}

Finally, we need the following version of several 
similar and well known inequalities, see, 
for example,~\cite[Section 8.2]{IwKow}. 

\begin{lemma}
\label{lem:Sum||}
For any integers $w$, $v$, $H$ and $Z$  with 
$$
\gcd(v,w) = 1, \qquad H \ge 1, \qquad 2Z > v \ge 1
$$
and a real $\alpha$ 
with 
$$
\left| \alpha - \frac{w}{v}\right| < \frac{1}{2Zv}
$$
we have
$$
\sum_{z=1}^Z \min\left\{H, \frac{1}{\dmod{\alpha z}{1}}\right\} \ll
HZv^{-1} + Z \log v.
$$
\end{lemma}

\begin{proof} First we estimate the contribution from $z \equiv 0 \pmod v$
trivially as $HZ/v$. For the remaining $z$ we note that
$$
\dmod{ \alpha z}{1}\ge \dmod{ \frac{wz}{v}}{1} - \frac{z}{2Zv} 
\ge \frac{1}{2} \dmod{ \frac{wz}{v}}{1}
$$
as 
$$
\dmod{\frac{wz}{v}}{1}  \ge \frac{1}{v} \mand \frac{z}{2Zv} \le \frac{1}{2v} .
$$
Therefore, the contribution from $z \not\equiv 0 \pmod v$ can be estimated as
\begin{equation*}
\begin{split}
\sum_{\substack{z=1\\z \not \equiv 0\pmod v}}^Z \min\left\{H, 
\frac{1}{\dmod{\alpha z}{1}}\right\} & \ll 
\sum_{\substack{z=1\\z \not \equiv 0\pmod v}}^Z \dmod{\frac{wz}{v}}{1}^{-1}\\
& \le \(Z/v + 1\) \sum_{z=1}^{v-1} \dmod{\frac{wz}{v}}{1}^{-1}.
\end{split}
\end{equation*}
Since $\gcd(v,w) = 1$ and $v \le 2Z$ this simplifies as 
\begin{equation*}
\begin{split}
\sum_{\substack{z=1\\z \not \equiv 0\pmod v}}^Z \min\left\{H, \frac{1}
{\dmod{\alpha z}{1}}\right\} & \ll
Zv^{-1}\sum_{z=1}^{v-1} \dmod{\frac{z}{v}}{1}^{-1} \\
& \ll
Z\sum_{1 \le z \le v/2} \frac{1}{z}  \ll Z\log v,
\end{split}
\end{equation*}
which concludes the proof.
\end{proof}

\subsection{Coefficients  of exceptional polynomials} 
\label{sec:Coeff Except}

Our first result uses in an essential way the argument 
 of the proof of~\cite[Theorem~3]{CGOS},

\begin{lemma}
\label{lem:NFIJ-2} Let 
$$
F(X) = \sum_{i=0}^\ell A_i X^i \in \F_p[X]
$$
be a polynomial of degree $\ell$. 
Assume that  for some  $\rho \le 1$  and
intervals $\cI = \{1, \ldots, H\}$ 
and $\cJ = \{1, \ldots, K\}$ with $1 \le H, K < p$ we have
$$
N_F(\cI,\cJ) \ge  \max\{(2\ell +1), \rho H\}.
$$
Then 
$$
A_i v \equiv u_i \pmod p
$$
for some  integers 
$$
1 \le v \ll  \rho^{-\ell(\ell+1)/2} \mand  u_i \ll \rho^{-\ell(\ell-1)/2} K H^{\ell-i}, 
\ i =0, \ldots, \ell.
$$
\end{lemma}

\begin{proof} As we have mentioned, we follow the proof of~\cite[Theorem~3]{CGOS}.


Let $N = N_F(\cI,\cJ)$. 
Since $N \ge 2(\ell +1)$, there exist $\ell+1$  pairs $(x_1,y_1),\ldots
,(x_{\ell+1},y_{\ell+1})$ such that $x_1,\ldots ,x_{\ell+1}$ lie in an
interval $\widetilde \cI$ of length $2(\ell+1)H/N$ and $y_1, \ldots, y_\ell \in \cJ$.

Now, we consider the system of congruences
$$
\left\{
\begin{array}{lll}
A_0+A_1x_1+\ldots +A_\ell x_1^\ell&\equiv &y_1\pmod p,\\
&\ldots & \\
A_0+A_1x_{\ell+1}+\ldots +A_\ell x_{\ell+1}^\ell &\equiv &y_{\ell+1}\pmod p.
\end{array}\right.
$$
The determinant $v$ of this system is the determinant of a Vandermonde matrix:
$$
v=\det \(\begin{matrix} 1 & x_1 &\ldots & x_1^\ell\\ \ldots & \ldots &
\ldots &\ldots \\1 & x_{\ell+1} &\ldots & x_{\ell+1}^\ell\end{matrix}\)=
\prod_{1\le i<j\le d+1}(x_j-x_i).
$$
Note that $v\not\equiv 0\pmod p$. Thus, we have that
\begin{equation}
\label{eq:vAiui}
vA_i\equiv u_i\pmod p,
\end{equation} 
where
\begin{equation}\label{ui}
\begin{split}
u_i&=\det \(\begin{matrix} 1 &\ldots
&x_1^{i-1}& y_1 & x_1^{i+1}&\ldots & x_1^\ell\\ \ldots & \ldots &
\ldots &\ldots \\1 & \ldots & x_{\ell+1}^{i-1}& y_{\ell+1} & x_{\ell+1}^{i+1}
&\ldots & x_{\ell+1}^\ell\end{matrix}\)\\
& =\sum_{j=1}^{\ell+1}
(-1)^{i+j}y_jV_{ij}
\end{split}
\end{equation}
and $V_{ij}$ is the determinant of
the matrix obtained from the Vandermonde matrix after removing the
$j$-th row and the $i$-th column, $i =0, \ldots, \ell$, $j =1, \ldots, \ell +1$.

It is easy to see that for each $i = 0, \ldots, \ell$, the 
determinant 
$V_{ij}$ is a
polynomial in $\ell$ variables $x_\nu$, $1 \le \nu \le \ell+1$, $\nu \ne j$, of degree 
$\ell(\ell+1)/2-i$, which
vanishes when $x_r=x_s$ for distinct $r$ and $s$.
Thus
$$
V_{ij}=W_i(x_1,\ldots, x_{j-1},x_{j+1},\ldots
,x_{\ell+1}) \prod_{\substack{1\le r<s\le \ell+1\\
r,s\ne j}}(x_s-x_r),
$$
where $W_i$ is a polynomial (that does not depend on $F$ or $p$)
of degree $\ell-i$.
Therefore, we have 
$$
V_{ij}\ll (H/N)^{\ell(\ell-1)/2}H^{\ell-i} \le \rho^{-\ell(\ell-1)/2}H^{\ell-i}.
$$ 
This
estimate  and~\eqref{ui} together imply the bound
\begin{equation}
\label{eq:ui}
u_i \ll \rho^{-\ell(\ell-1)/2} K H^{\ell-i}, \qquad  i =0, \ldots, \ell.
\end{equation}
On the other hand, it is clear that
\begin{equation}
\label{eq:v}
v\ll  (H/N)^{\ell(\ell+1)/2} \le \rho^{-\ell(\ell+1)/2}.
\end{equation}
Then, comparing~\eqref{eq:vAiui} with~\eqref{eq:ui} and~\eqref{eq:v}
we conclude the proof. 
\end{proof}

We note that if $s$ and $\rho$ are fixed the expressions given by the
example~\eqref{eq:exam}
and those of Lemma~\ref{lem:NFIJ-2} differ by a factor $H^\ell$ and it is 
certainly natural to try to reduce this gap. 

We now obtain yet another estimate on $N_F(\cI,\cJ)$ that depends 
only on the behaviour of the leading coefficient. 

We recall the definition of $\kappa(\ell)$ used 
in  Lemma~\ref{lem:NFIJ-1}. Since for $H^\ell > K$ we can use the
bound of Lemma~\ref{lem:NFIJ-1}, we assume that 
\begin{equation}
\begin{split}
\label{eq:H K}
2 \ell!H^{\ell-1} \le K. 
\end{split}
\end{equation}
Then the argument of the proof of~\cite[Theorem~5]{CCGHSZ} 
leads to the following result:

\begin{lemma}
\label{lem:NFIJ-3} Let 
$$
F(X) = \sum_{i=0}^\ell A_i X^i \in \F_p[X]
$$
be a polynomial of degree $\ell$. 
Assume that  for some $\rho \le 1$ and
intervals $\cI = \{1, \ldots, H\}$ 
and $\cJ = \{1, \ldots, K\}$ with $1 \le H, K < p$
satisfying~\eqref{eq:H K} we have
$$
N_F(\cI,\cJ) \ge \rho H.
$$
Then either 
$$
\rho \ll \min\{K/p, H^{-1/2^{\ell-1}}p^{o(1)}\},
$$
or
$$
A_\ell v \equiv  u \pmod p
$$
for some  integer
$$
1 \le v \le \rho^{-2^{\ell-1}}p^{o(1)} \mand
u\ll H^{-\ell+1}K.
$$
\end{lemma}

\begin{proof} 

Let $N = N_F(\cI,\cJ)$. We apply Lemma~\ref{lem:ET small int} to the sequence
of fractional parts $\gamma_n = \{F(t)/p\}$, $t \in \cI$,
with
$$\alpha = 1/p, \qquad \beta = K/p, \qquad R = \fl{p/K}.
$$
Without loss of generality, we can assume that $K < p/2$ as otherwise
the bound is trivial. Then we have $R\ge 1$ and thus
$$
\frac{1}{R} +\min\{\beta - \alpha, 1/r\} \ll \frac{K}{p}
$$
for $r =1, \ldots, R$, we derive
$$
N \ll \frac{HK}{p} + \frac{K}{p} \sum_{r=1}^R
\left|\sum_{t=1}^H \exp(2 \pi i r F(t)/p)\right|.
$$
Therefore, by Lemma~\ref{lem:Weyl}, we have
\begin{equation*}
\begin{split}
N \ll \frac{HK}{p} &+  \frac{H^{1-\ell/2^{\ell-1}}K}{p}   \\
  \times &\sum_{r=1}^R
\(\sum_{-H < t_1 , \ldots, t_{\ell-1}  < H}
\min\left\{H, \dmod{\frac{A_\ell}{p}\ell! r 
t_1   \cdots  t_{\ell-1}}{1}^{-1}\right\}\)^{1/2^{\ell-1}}.
\end{split}
\end{equation*}
Now, separating from the sum over $t_1,\ldots,t_{\ell-1}$
the terms with $t_1 \cdots t_{\ell-1} = 0$, 
(giving a total contribution $O(H^{\ell-1})$), we obtain
\begin{equation}
\begin{split}
\label{eq:N W}
N & \ll \frac{HK}{p} +  \frac{H^{1-\ell/2^{\ell-1}}K}{p} R(H^{\ell-1})^{1/2^{\ell-1}}
+ \frac{H^{1-\ell/2^{\ell-1}}K}{p}  W \\
& \ll \frac{HK}{p} +  H^{1-1/2^{\ell-1}} 
+ \frac{H^{1-\ell/2^{\ell-1}}K}{p}  W ,
\end{split}
\end{equation}
where
$$
W = \sum_{r=1}^R
\(\sum_{0 < |t_1|, \ldots,  |t_{\ell-1}| < H}
\min\left\{H, \dmod{\frac{A_\ell}{p}\ell! r t_1  \cdots
t_{\ell-1}}{1}^{-1}\right\}\)^{1/2^{\ell-1}}.
$$
Hence, recalling the choice of $R$, we derive
\begin{equation}
 \label{eq:J and W}
N \ll  {HK}{p} +     H^{1-1/2^{\ell-1}} + \frac{H^{1-\ell/2^{\ell-1}}K}{p}   W.
\end{equation}
The H\"{o}lder inequality implies the bound
 \begin{equation*}
\begin{split}
W^{2^{\ell-1}} \ll R^{2^{\ell-1}-1}
& \sum_{r=1}^R  \\
&\sum_{0 < |t_1|, \ldots,  |t_{\ell-1}| < H}
\min\left\{H, \dmod{\frac{A_\ell}{p}\ell! r t_1   \cdots
t_{\ell-1}}{1}^{-1}\right\}.
\end{split}
\end{equation*}
Collecting together the terms with the same value of
$$
z = \ell! r t_1   \cdots  t_{\ell-1}
$$ 
and using the well-known bound on the divisor function,
we conclude that
$$
W^{2^{\ell-1}} \ll R^{2^{\ell-1}-1} p^{o(1)}
\sum_{0 < |z| < \ell! H^{\ell-1} R}
\min\left\{H, \dmod{\frac{A_\ell}{p}z }{1}^{-1}\right\}.
$$

Let us set $Z = \ell!H^{\ell-1} R$
Note that our assumption~\eqref{eq:H K} implies that.
$$
2Z  < p.
$$
By the Dirichlet approximation theorem we can write 
$$
\left|\frac{A_\ell}{p} - \frac{w}{v}\right| \le \frac{1}{2Zv}
$$
for some integers $w$ and $1 \le v < 2Z$. Note that this immediately
implies that $|A_\ell v - pw| \le p/2Z$.
Hence
$$
A_\ell v  \equiv u \pmod p
$$
for some integer $u$ with
$$
|u| \le \frac{p}{2Z} \ll H^{-\ell+1}K.
$$

It now remains to estimate $\rho$ or $v$. 

We apply Lemma~\ref{lem:Sum||} with $Z = \ell!H^{\ell-1} R $, and we derive:
\begin{equation*}
\begin{split}
W^{2^{\ell-1}} &\ll R^{2^{\ell-1}-1} p^{o(1)}\(H^\ell R v^{-1} + H^{\ell-1} R\)\\
& = R^{2^{\ell-1}} p^{o(1)}\(H^\ell  v^{-1} + H^{\ell-1}\), 
\end{split}
\end{equation*}
which after the substitution in~\eqref{eq:N W} implies
$$
N  \ll \frac{HK}{p} + H^{1-1/2^{\ell-1}} + \frac{HK Rv^{-1/2^{\ell-1}}}{p^{1+o(1)}}
 +  \frac{H^{1-\ell/2^{\ell-1}}K}{p^{1+o(1)}} R(H^{\ell-1})^{1/2^{\ell-1}}.
$$
Noting that the last term, up to the factor $p^{o(1)}$ coincides with the second term, 
we now obtain
$$
N  \ll \frac{HK}{p} + H^{1-1/2^{\ell-1}}p^{o(1)}+ H v^{-1/2^{\ell-1}}p^{o(1)}.
$$
Hence we either have
$$
\rho H   \ll \frac{HK}{p} + H^{1-1/2^{\ell-1}}p^{o(1)}
$$
or 
$$
\rho H  \ll  H v^{-1/2^{\ell-1}}p^{o(1)},
$$
which concludes the proof. 
\end{proof}

Note that the dependence on $\rho$ in the bound on $v$ of 
Lemma~\ref{lem:NFIJ-3} is worse that that of Lemma~\ref{lem:NFIJ-2}.
However, since we are mostly interested in large values of $\rho$, for example,
the extreme case $\rho =1$ is of special interest, Lemma~\ref{lem:NFIJ-3} 
limits the possible values of $A_\ell$  as $O( H^{-\ell+1}K)$, 
which is  stricter than the 
bound  $O\(K\)$ that follows from the estimates 
of  Lemma~\ref{lem:NFIJ-2} on $v$ and $u_\ell$. 

\section{Lattices and Polynomials}

\subsection{Background on lattices}
\label{sec:Lat CVP}

As in~\cite{BV1,BV2}, and then in~\cite{Shp1,ShpWint}, 
our results rely on 
some lattice algorithms. 
 We therefore review some relevant results and definitions, we refer 
to~\cite{ConSlo,GrLoSch,GrLe} for more details and the general theory.

Let $\{{\vb}_1,\ldots,{\vb}_s\}$ be a set of $s$ linearly independent vectors in
${\R}^s$. The set
of  vectors
$$
L=\{\vz  \ : \ \vz=\sum_{i=1}^s\ c_i\vb_i,\quad c_1, \ldots, c_s\in\Z\}
$$
is called an $s$-dimensional full rank lattice. 

The set $\{ {\vb}_1,\ldots,
{\vb}_s\}$
is called a {\it basis\/}
of $L$.

The volume of the parallelepiped defined by the 
vectors ${\vb}_1,\ldots, {\vb}_s$ is called {\it 
the volume\/} of the lattice and denoted by $\vol(L)$. 
Typically, lattice problems are easier when the Euclidean 
norms of all basis vectors are close to $\vol(L)^{1/s}$. 

One of the most fundamental problems in this area is 
 the {\it closest vector problem\/},  CVP:
given a basis of a lattice $L$ in $\R^s$ and a target vector $\vec{u} \in \R^s$,
find a lattice vector $\vec{v} \in L$ which minimizes
the Euclidean norm $\|\vec{u}-\vec{v}\|$ among all lattice vectors.
It is well know that  CVP is  {\bf NP}-hard when the 
dimension $s\to \infty$
(see~\cite{Ngu,NgSt1,NgSt2,Reg} for references). 

However, its approximate 
version admits a deterministic polynomial time algorithm which goes back to 
the lattice basis reduction algorithm of Lenstra, Lenstra
and  Lov{\'a}sz~\cite{LLL}, see also~\cite{Ngu,NgSt1,NgSt2,Reg} 
for possible improvements and further references.

We remark that all lattices that appear in this paper are 
of fixed dimension, so instead of using the above approximate 
CVP algorithms,
we can simply use the following result of 
Micciancio and  Voulgaris~\cite{MicVou}. 

Let $\|\vz\|$ denote the standard Euclidean norm in ${\R}^s$.

\begin{lemma}
\label{lem:CVP}
For any fixed $s$, there exists a deterministic algorithm which,
for given a lattice $L$ generated by an integer vectors 
$$
{\vb}_1,\ldots,{\vb}_s \in \Z^s
$$
 and a vector $\vr =(r_1, \ldots, r_s) \in{\Z}^s$,
in time polynomial in 
$$
\sum_{i=1}^s \(\log \|b_i\| + \log |r_i| + 1\)
$$
finds a lattice vector
$\vv = (v_1, \ldots, v_s) \in L$ satisfying 
$$
\|\vv-\vr\| =\min \left\{\|\vz-\vr\|~:~\vz  \in L\right\}.
$$
\end{lemma}



\subsection{Lattices and polynomial approximations}
\label{sec:Lap Poly}

 For  $t_1, \ldots, t_d \in \F_p$, integer $n> k\ge 1$, we set $m = n+1-k$ and
we denote by $\cL_{k,n,p}\(t_1, \ldots, t_d\)$ the $d+m$-dimensional lattice
generated by the rows of the following $(d +m)\times (d+m)$-matrix
\begin{equation}
\label{eq:matrix}
 \(\begin{array}{lllllllll}
 p & 0 & \ldots &  0   & 0&  \ldots  & 0\\
 0 & p & \ldots &  0   &0&  \ldots  & 0\\
 \vdots & {} & \ddots& {} &\vdots & { } & \vdots\\
 0 & 0 &  \ldots & p &  0 & \ldots & 0  \\
 t_1^{k} & t_2^{k} &  \ldots & t_d^{k} & 2\Delta/p & \ldots & 0 \\
  t_1^{k+1} & t_2^{k+1} &  \ldots & t_d^{k+1} & 0& \ldots & 0 \\
 \vdots & {} & { } & {} & \vdots & \ddots & \vdots\\
 t_1^{n} & t_2^{n} & \ldots & t_d^{n} &   0& \ldots & 2\Delta/p\\
 \end{array}\).
\end{equation}

The following result is a  generalization of
several previous results of similar flavour, see~\cite{BV1,Shp1,ShpWint}.

\begin{lemma}
\label{lem:Latt-Pol}
Let $\delta>0$ be fixed and 
let $p$ be a sufficiently large prime.
Assume that some 
real numbers $D \ge \Delta \ge 1$ and integer numbers  $h\ge 1$, $k \ge 1$ satisfy 
$$
 D \le  \min\left\{ p^{1-\delta}, h^k p^{-\delta} \right\}.
$$
Then there is an integer $d$, depending only on $n$ and $\delta$, 
such that if 
$$
f(X) = \sum_{j=k}^n a_j X^{j} \in \F_p[X]
$$
and  $t_1, \ldots, t_d \in [-h,h]$ are
chosen uniformly and independently at random, 
then with
 probability at least $1- 1/p$ the following holds. 
For any vector $\vs = (s_1, \ldots, s_d, 0, \ldots, 0)$ with
$$
\dmod{f(t_i) - s_i}{p} \le \Delta, \qquad i=1, \ldots, d,
$$
all vectors
$$
\vv = (v_1, \ldots, v_{d}, v_{d+1}, \ldots, v_{d+m}) 
  \in \cL_{k,n,p}\(t_1, \ldots, t_d\)
$$
satisfying
\begin{equation}
\label{eq:vi si D}
\|\vv - \vs\|\le  D,
\end{equation}
are of the form
$$
\vv = \( f(t_1), \ldots,  
f(t_d),  a_{k}\Delta/p, \ldots, a_n\Delta/p\). 
$$
\end{lemma}

\begin{proof} 
Let $\cP_{k,f}$ denote the set of $p^{n+1-k}-1$ polynomials  
\begin{equation}
\label{eq:g-tilde}
g(X) = \sum_{j=k}^n b_j X^{j} \in \F_p[X]
\end{equation}
with $g \ne f$.

We see from Lemma~\ref{lem:NFIJ-1} (taking into account that 
$\kappa(\ell)$ is a monotonically non-decreasing function of 
$\ell$) that
for any polynomial $g \in \cP_{k,f}$
the probability $\rho\(g\)$ that 
$$
\dmod{f(t) - g(t)}{p}  \le 2D
$$
for $t \in [-h,h]$ selected uniformly at random is
$$
\rho\(g\) =   h^{o(1)} \((D/p)^{1/2\kappa(n)} + (D/h^k)^{1/2\kappa(n)}\) \le p^{-\eta}
$$
provided that $p$ is large enough, where $\eta> 0$ depends only on $\delta$ and $n$.

Therefore, for any  $g \in \cP_{k,f}$,
\begin{equation*}
\begin{split}
\Pr\left[\exists i\in [1,d]~:~\dmod{f(t_i)-g(t_i)}{p} >   2D\right]
= 1 - \rho\(g\)^d\ge 1 -p^{-d \eta}, 
\end{split}
\end{equation*}
where the probability is taken over $t_1, \ldots, t_d \in \F_p$
chosen uniformly and independently at random.

Since $\# \cP_{k,f} < p^{n-k+1} \le p^n$,  we obtain 
\begin{equation*}
\begin{split}
\Pr\left[\forall  g\in \cP_{k,f}, \ 
\exists i\in [1,d]~:~\dmod{f(t_i)-g(t_i)}{p} >  2D\right]&\\
\ge 1 - p^{n -\eta d} > 1 & - p^{-1}  
\end{split}
\end{equation*}
for $d > (n+1) \eta^{-1}$, 
provided that $p$ is large enough.

We now fix some integers $t_1, \ldots, t_d$ with
\begin{equation}
\label{eq:Cond}
\min_{g \in \cP_{k,f}} \  \max_{i\in [1,d]}
 \dmod{f(t_i)-g(t_i)}{p} >  2D.
\end{equation}

Let $\vv \in \cL_{k,n,p}\(t_1, \ldots, t_d\)$ be a lattice point
satisfying~\eqref{eq:vi si D}. 

Clearly, for a vector $\vv\in \cL_{k,n,p}\(t_1, \ldots, t_d\)$,
there are some  integer $b_k, \ldots, b_n,z_1,\ldots,z_d$ such that
$$
\vv = \(\sum_{j=k}^n b_j t_1^{j} - z_1 p,\ldots,\sum_{j=k}^n b_j t_d^{j} - z_d
p,b_k\Delta/p, \ldots, b_n\Delta/p\).
$$

Suppose that $b_j \not \equiv a_j \pmod p$ for some $j =k, \ldots, n$
and let $g$ be given by~\eqref{eq:g-tilde}.
In this case we have
\begin{equation*}
\begin{split} 
\|\vv - \vs\| &\ge  \max_{i \in [1,d]}
 \dmod{g(t_i) - s_i}{p} \\ 
&  \ge   \max_{i \in [1,d]} \(
\dmod{f(t_i) - g(t_i)}{p}  -  \dmod{s_i - f(t_i)}{p}\)      >   2D -  \Delta  \ge D
\end{split}
\end{equation*}
that contradicts~\eqref{eq:vi si D}.

Now, if  $b_j \equiv a_j \pmod p$, $j =k, \ldots, n$, then for all
$i=1,\ldots,d$ we have
$$
\left|\sum_{j=k}^n b_j t_i^{j} - z_i p\right| = \dmod{\sum_{j=k}^n b_j
t_i^{j}}{p} = g(t_i),
$$
since otherwise there
is 
$i\in [1,d] $ such that $|v_i - g(t_i)| \ge p$ and thus
$|v_i - s_i | \ge p - D$, which contradicts~\eqref{eq:vi si D} again. 
Hence $\vv$ is of the desired form.

As we have seen, the condition~\eqref{eq:Cond} holds with probability
exceeding $1- 1/p$ and the result follows.
\end{proof}

\section{Main Results}
 
\subsection{Polynomial interpolation problem}
\label{sec:Poly Int}

We are ready to prove the main result.
We follow the same arguments as in the proof 
of~\cite[Theorem~1]{BV1} that have also been 
used  for several other similar problems,
see~\cite{Shp1,Shp2,ShpWint}.

\begin{theorem}
\label{thm:PolAppr} 
Let $\varepsilon>0$ be fixed and 
let $p$ be a sufficiently large prime.
Assume that a  real $ \Delta \ge 1$ and  integers $p/2 > h\ge 1$, $k \ge 1$ satisfy 
$$
h^k >   \Delta p^{\varepsilon} \mand  \Delta < p^{1-\varepsilon}.
$$
There exists an integer $d$, depending only on $n$ and $\varepsilon$
and a  deterministic polynomial time algorithm $\cA$
such that for any polynomial
$$
f(X) = \sum_{j=k}^n a_j X^{j} \in \F_p[X]
$$
given integers $t_1, \ldots, t_d \in [-h,h]$ and 
$s_1, \ldots, s_d \in \F_p$ with
$$
\dmod{f(t_i) - s_i}{p} \le \Delta, \qquad i=1, \ldots, d,
$$
its output satisfies
$$
\Pr_{t_1, \ldots, t_d \in  [-h,h]} \left[
\cA\(t_1, \ldots, t_d; s_1, \ldots, s_d \)  = (a_k, \ldots, a_n)\right] \ge 1- 1/p
$$
if  $t_1, \ldots, t_d$ are
chosen uniformly and independently at
random from  $[-h,h]$.
\end{theorem}

\begin{proof}  We refer to
the first $d$ vectors in  the
matrix~\eqref{eq:matrix} as $p$-vectors and we refer to the other $m=n+1-k$ vectors 
as power-vectors.

Let us consider the vector $\vs =(s_1, \ldots, s_d, s_{d+1}, \ldots, s_{d+m})$
where
$$
s_{d+j} =0, \qquad j =1, \ldots, m.
$$
Multiplying the  $j$th power-vector  of the
matrix~\eqref{eq:matrix}
 by $a_j$ and subtracting a
 certain multiple of the $j$th $p$-vector, $j =k, \ldots, n$, we obtain a lattice point
\begin{equation*}
\begin{split} 
{\vw}_{f}&=  (w_1,\ldots,w_{d+m}) \\
& = (f(t_1), \ldots, f(t_d), 
a_1\Delta/p, \ldots, a_m\Delta/p)  \in \cL_{k,n,p}\(t_1, \ldots,
t_d\), 
\end{split}
\end{equation*}
and thus 
$$
\dmod{w_i - s_i}{p} \le  \Delta,  \qquad i = 1, \ldots, d+m. 
$$

Therefore,
\begin{equation}
\label{eq:wi si}
\| \vw - \vs\| \le  \sqrt{d+m} \Delta.
\end{equation}
Now we can use Lemma~\ref{lem:CVP}  to find in polynomial time
time a lattice vector
$\vv = (v_1,\ldots,v_d,v_{d+1},  \ldots, v_{d+m} )\in  \cL_{k,n,p}\(t_1, \ldots, t_d\)$ 
such that
$$
\|\vv- \vs\|
 =\min
\left\{\| \vz - \vs\| ~:~   \vz \in   \cL_{k,n,p}\(t_1, \ldots, t_d\)\right\} \le \| \vw - \vs\|.
$$
Hence, we conclude from~\eqref{eq:wi si} that
$$
\| \vw - \vv\| \le  2\sqrt{d+m} \Delta.
$$
Applying Lemma~\ref{lem:Latt-Pol} with $D = \sqrt{2(d+m)}\Delta$ and,
say $\delta = \varepsilon/2$) 
we see that $\vv={\vw}_{f}$ with probability
at least $1- 1/p$, and therefore the coefficients of  $f$ can be recovered in
polynomial time.
\end{proof}

\subsection{Polynomial approximation problem}
\label{sec:Poly Appr}

As we have mentioned, it is impossible to recover a ``complete'' polynomial 
of degree $n$ from the approximations~\eqref{eq:Approx}
for $t \in [-h,h]$ for a small value of $h$. 
However, it is interesting to estimate the number of possible 
``false'' candidates  $\tf \in \F_p[X]$, which for many 
$t\in [-h,h]$ take values close to those of $f$.

Namely, for a polynomial $f \in \F_p[X]$, 
and real positive parameter $\rho\le 1$ we define 
$\# \cM_f(\rho, h,\Delta)$ as the set of  polynomials $\tf \in \F_p[X]$
with 
$$
\dmod{\tf(t) - f(t)}{p} \le \Delta
$$
for at least $\rho (2h+1)$ values of $t \in [-h,h]$. 
For example, if for some fixed $\varepsilon$ we have
$$
h^{n} \ge   \Delta p^{\varepsilon} \mand  \Delta < p^{1-\varepsilon}, 
$$ 
the by Lemma~\ref{lem:NFIJ-1} we have $\cM_f(\rho, h,\Delta) = \emptyset$
unless $\rho \le  p^{-\varepsilon/2 \kappa(n)+o(1)}$.
So we now consider slightly smaller values of $h$.

\begin{theorem}
\label{thm:PolAppr-Imposible} 
Let $\varepsilon>0$ be fixed and 
let $p$ be a sufficiently large prime. 
Assume  real $ \Delta \ge 1$ and  integers $p/2 > h \ge 1$, $n \ge 1$ satisfy 
$$
h^{n-1} \le   \Delta p^{-\varepsilon} \mand  \Delta < p^{1-\varepsilon}, 
$$ 
for some fixed $\varepsilon > 0$. 
Then for $\rho \ge \max\{\Delta p^{-1}, h^{-1/2^{n-1}}\} p^{\varepsilon}$
we have 
$$
\# \cM_f(\rho, h,\Delta) \ll \rho^{-2^{n-1} -n(n^2 +1)/2}
 \Delta^{n+1} h^{(n^2-n+2)/2} p^{o(1)}. 
$$
\end{theorem}

\begin{proof}  For every $\tf \in \cM_f(\rho, h,\Delta)$, we apply  
Lemma~\ref{lem:NFIJ-3} to estimate the number of 
possible values of the leading coefficient of $f-\tf$ as 
$\rho^{-2^{n-1}}\Delta h^{-n+1}p^{o(1)}$.
Then we also use Lemma~\ref{lem:NFIJ-2} to estimate the 
number of possible values for the other  coefficients  of $f-\tf$
as  
\begin{equation*}
\begin{split} 
\rho^{-n(n+1)/2} \prod_{i=0}^{n-1}  \rho^{-n(n-1)/2} \Delta h^{n-i}
& = \rho^{-n (n+1)/2 - n^2(n-1)/2}\Delta^{n} h^{n (n+1)/2}\\
& = \rho^{-n(n^2 +1)/2}\Delta^{n-1} h^{n (n+1)/2}
\end{split}
\end{equation*}
which concludes the proof. 
\end{proof}

The case of $\rho = 1$, when the recovery of $f$ from
the approximations~\eqref{eq:Approx}
with $t \in [-h,h]$  is impossible regardless of the 
number of queries and the complexity of algorithm 
is certainly of special interest. 
We denote  $\overline\cM_f(h,\Delta) =  \cM_f(1, h,\Delta)$.

Under the conditions of $h$ and $\Delta$ of Theorem~\ref{thm:PolAppr-Imposible} 
we have
$$
\overline\cM_f(h,\Delta) \ll  \Delta^{n} h^{n(n+1)/2} p^{o(1)}.
$$
Clearly if the coefficients of $\tf \in \F_p[X]$ are sufficiently 
close to those of $f$ then $\tf \in  \overline\cM_f(h,\Delta)$.
This argument easily leads to the lower bound
$$
 \overline\cM_f(h,\Delta) \gg  \Delta^{n+1} h^{-n(n+1)/2}.
$$

We now show that this is not necessary and in particular 
the set $ \overline\cM_f(h,\Delta)$ 
contains many polynomials that ``visually'' look very differently from $f$.

\section{Constructions of Exceptional Polynomials}

\subsection{Flat polynomials}
\label{sec:Flat Poly}

Let us take any integer $v$ and non-zero polynomial $\Psi_i(X) \in \Z[X]$
of degree $i=1, \ldots, n$ and with coefficients in the range $[0, v-1]$
such that 
$$
\Psi_i(t) \equiv 0 \pmod v, \qquad t = 0, \ldots, v-1.
$$
For examples if $v$ is a square-free integer, 
we can construct such a polynomial via the Chinese Remainder 
Theorem from polynomials of the form $\psi(X)(X^r-X) \in \F_r[X]$
for each prime divisor $r \mid v$. 

Alternatively, given an integer $v\ge 1$, for any $i$ with
\begin{equation}
\label{eq:v i!}
v \mid i!
\end{equation}
one can consider take $\Psi_i(X) = A_iX(X-1)\ldots (X-i+1)$
with an integer $A_i$ that is relatively prime to $v$. 
We note that the function $S(v)$, defined as the smallest $i$ with the 
property~\eqref{eq:v i!}
is called the {\it Smarandache function\/} and has been extensively studied
in the literature, see~\cite{Iv} and references therein.

Now, for any sufficiently small integers $u_i$ we define the coefficients 
$b_i$
from the congruences $b_i v \equiv u_i \pmod p$, $i =1, \ldots n$,
and consider the polynomial
$$
F(X) \equiv  \sum_{i=1}^n b_i\Psi_i(X) \in \F_p[X].
$$
It is easy to see that the coefficients of $F(X)$ 
can be rather large compared to $p$. Furthermore, 
for every $t \in [-h,h]$ as the fractions
$\Psi_i(t)/v$, $i =1, \ldots, n$, take integer values 
of size $O(h^i)$, we have 
$$
F(t)  \equiv  \sum_{i=1}^n u_i \frac{\Psi_i(t)}{v} \equiv U \pmod p,
$$
where 
$$
U \ll  \sum_{i=1}^n u_i h^i.
$$
Varying the parameter $v$, the polynomials $\Psi_i$ and the 
integers $u_i$, $i =1, \ldots n$, (and also introducing a small
non-zero constant coefficient $b_0$) one can get a large family 
of such polynomials $F$ that remain ``flat'' on short intervals. 
In particular, one can choose $v$ to be a product of many small
primes. 
This argument can easily  be made more precise with explicit 
constants, however we instead present an illustrative 
example of such a polynomial.

Let us fix the prime
\begin{equation}
\label{eq:prime}
\begin{split}
p=138501785460241506&76274172131557249442552\\
& 857086506417208905998552087,
\end{split} 
\end{equation}
We also set
$h=2^{31}$, $\Delta=\lfloor p/2^{33} \rfloor$
and  for $i=1,\ldots,5$,  choose integers $0< A_i < p$ such that
$A_i i! \equiv 1 \pmod p$:
\begin{align*}
A_1&=1\\
A_2&= 6925089273012075338137086065778624721276428543\\ 
& \qquad\qquad\qquad\qquad\qquad\qquad\qquad\qquad\qquad 253208604452999276044 \\
A_3&= 1154181545502012556356181010963104120212738090\\ 
& \qquad\qquad\qquad\qquad\qquad\qquad\qquad\qquad\qquad 5422014340754998793406\\
A_4&= 9810543136767106729027538593186385021808273769\\ 
& \qquad\qquad\qquad\qquad\qquad\qquad\qquad\qquad\qquad 608712189641748974395 \\
A_5&= 1962108627353421345805507718637277004361654753\\ 
& \qquad\qquad\qquad\qquad\qquad\qquad\qquad\qquad\qquad 921742437928349794879
\end{align*}
and pick numbers $c_1,\ldots,c_5$ randomly such that $0< c_i h^i/i! < \Delta$,
for example:
\begin{align*}
c_1 &=728236268016142987379676454561254599761666551820 \\
c_2 &=118901258278655898193398330486974890011 \\
c_3 &= 80243828316297659193667769559\\
c_4 &=177312506479764141124 \\
c_5 &= 210305526612,
\end{align*}
and let 
$$
F(X) = \sum_{i=1}^5 c_i A_i \prod_{k=0}^{i-1} (X-k).
$$

The  plot on Figure~\ref{fig:flat},  shows
that $F(x)<\Delta$ (where $\Delta$ is indicated by the horizontal line)
for $0\leq x<h$:
\begin{figure}[H]
\includegraphics[width=0.8\textwidth]{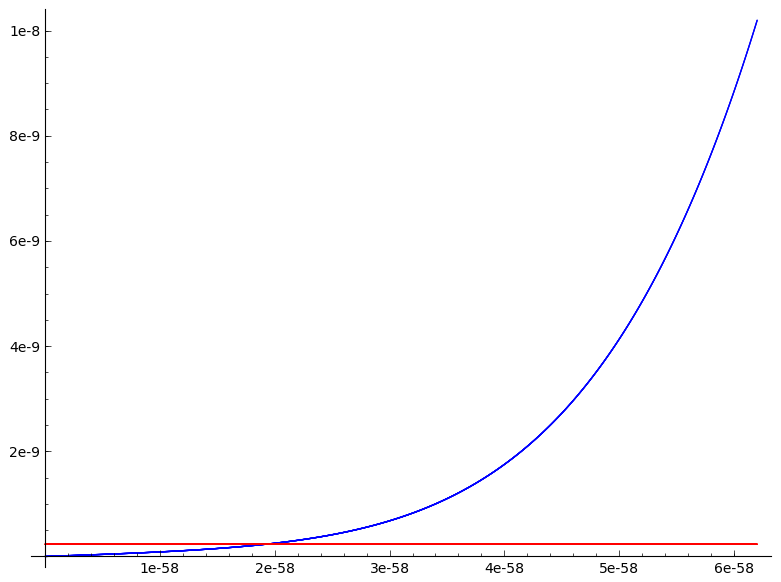}
\caption{
Plots of $F(x)/p$ versus $x/p$ for $0\leq x < 4 h$. }
\label{fig:flat}.
\end{figure}

The plot in Figure~\ref{fig:flat2} 
shows that $F(x)/p$ stays flat and close to zero for $x$ close to~0
and, after some transition period, appears to behave
completely randomly for $x \ge 300h$. 

\begin{figure}[H]
\includegraphics[width=0.8\textwidth]{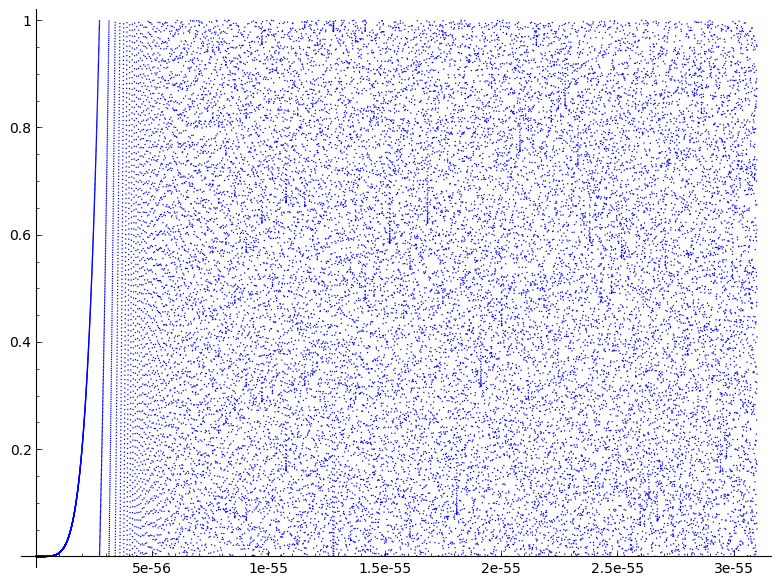}
\caption{
Plots of $F(x)/p$ versus $x/p$  for $0\leq x < 2000 h$. }
\label{fig:flat2}
\end{figure}

\subsection{Oscillating polynomials}
\label{sec:Oscill Poly}

A different type of polynomial with large coefficients that is small
in an interval can be constructed if $2^{n+1} \Delta > p$.
Thus this class is interesting only if the degree $n$ becomes a 
growing parameter, which is actually the case in the settings of HIMMO.

Let $f\in\Z [X]$ be polynomial of degree $n$ that, after reduction modulo $p$
takes values in the interval $[-\Delta,\Delta]$ for $x\in[-h,h]$. This means
that
$$ f(X) = d(X) + p c(X) $$
for integer valued functions $d$ and $c$, where for $x\in[-h,h]$, $d(x)\in
[-\Delta,\Delta]$.

Let $D$ denote the discrete difference operator, that is, for any function $A$,
$DA(x)=A(x+1)-A(x)$. We have, since $f$ is a polynomial of degree~$n$:
$$ 0=D^{n+1}f(x) = D^{n+1}d(x) + p D^{n+1} c(x),$$
so $D^{n+1}d(x) \in p\Z$.
On the other hand, since $d(x)\in[-\Delta,\Delta]$ for $x\in[-h,h]$, it holds
that
$D^{n+1}d(x)\in[-2^{n+1}\Delta,2^{n+1}\Delta]$ for $x\in[-h,h-n-1]$.
For $2^{n+1}\Delta < p$ it follows that $D^{n+1}D(x) =0$ for all
$x\in[-h,h-n-1]$,
so that after $n+1$-fold integration we have
$$
d(x) = \sum_{i=0}^n D^i d(0) \binom{x}{i}.
$$
This means that $d$ and $c$ are both polynomials of 
degree at most $n$ on the interval $[-h,h]$. 

For $2^{n+1}\Delta > p$ there are more possibilities. For instance, if 
$$ D^{n+1}d(x) = p (-1)^x \text{ for } x\in[-h,h-n-1],$$
we obtain, by integrating
\begin{align*}
 D^{n}d(x) &= D^n d(0) + \sum_{y=0}^{x-1} D^{n+1} d(y) = D^nd(0) +
p \frac{1-(-1)^x}{2}, \\
D^{n-1}d(x) &= D^{n-1} d(0) + \( D^n d(0) + p/2 \) \binom{x}{1}
-p \frac{1-(-1)^x}{4}, \\
&\vdots  \\
d(x) &= d(0) + \sum_{i=1}^n \( D^i d(0) - \(-\frac{1}{2}\)^{n+1-i} p
\) \binom{x}{i}\\
& \qquad \qquad \qquad \qquad \qquad  \qquad \qquad - (-1)^{n+1} p \frac{1 - (-1)^x}{2^{n+1}}.
\end{align*}
Defining 
$$ c(x) =  \sum_{i=0}^n \(-\frac{1}{2}\)^{n+1-i} \binom{x}{i}
- \(-\frac{1}{2}\)^{1+n} (-1)^x,$$
and $$f(x) = \sum_{i=0}^n D^id(0) \binom{x}{i},$$
we have a decomposition $f(x) = d(x) + p c(x)$, where, for integer $x$, these
three functions are integer valued. By choosing
$d(0)$ to be close to $0$ and, for $1\leq i\leq n$, 
$D^id(0)$ close to the residues of $\(-1/2\)^{1+n-i} p$ 
modulo $p$, we obtain a 
polynomial $f$ of degree~$n$ that, after reduction modulo~$p$, lies in
$[-\Delta,\Delta]$ for arguments in the interval $[-h,h]$. 
For the prime $p$ given by~\eqref{eq:prime}, with 
$n=5$, 
and
$$
d(0)=0 \mand D^id(0)\equiv\near{\(-1/2\)^{1+n-i} p} \pmod p, \ 1\leq i\leq 5,
$$
where $\near{\xi}$ denotes the closest integer (and defined arbitrary in 
case of a tie), 
we obtain the graph in Figure~\ref{fig:doubleflat}. Note that the left 
part of the graph in Figure~\ref{fig:doubleflat} 
looks like a plot of two curves, while in fact this is 
one curve, which exhibits a highly oscillatory behaviour. 
\begin{figure}[H]
\includegraphics[width=.9\textwidth]{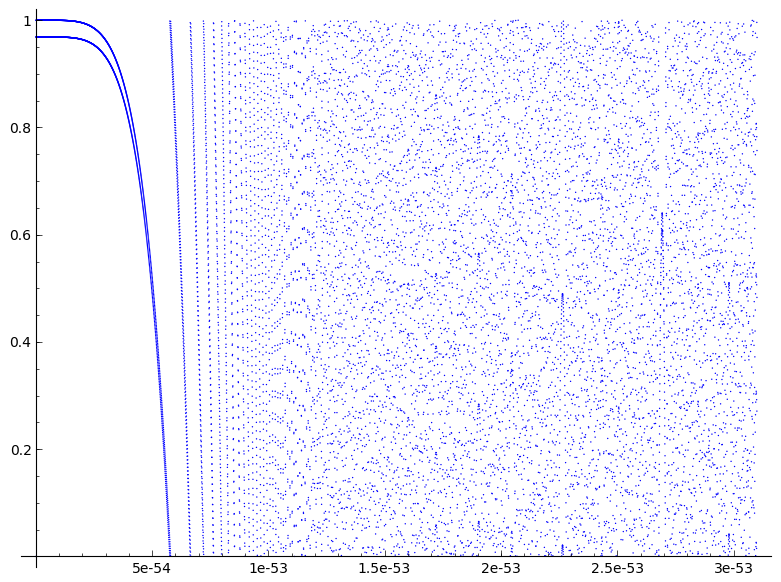}
\caption{Graph of $ f(x)/p$ versus $x/p$.}
\label{fig:doubleflat}.
\end{figure}

\section{Approximate recovery}
\label{sec:Poly App}

\subsection{Approach}

Using lattice reduction and rounding techniques, we are able to find a
polynomial $\widetilde{f}$ satisfying $\dmod{ u_t -
\widetilde{f}(t) }{p} \leq \Delta$ for
$t\in\{t_1,t_2,\ldots,t_d\}$, the set of observation points. The question is
whether $\widetilde{f}(t)$ approximates $f(t)$ also in all or at least 
some, non-observed points
$t\in[-h,h]\setminus\{t_1,\ldots,t_d\}$.
The lattice used in this technique is spanned by the rows of
$$
 \(\begin{array}{lllllllll}
 p & 0 & \ldots &  0   & 0&  \ldots  & 0\\
 0 & p & \ldots &  0   &0&  \ldots  & 0\\
 \vdots & {} & \ddots& {} &\vdots & { } & \vdots\\
 0 & 0 &  \ldots & p &  0 & \ldots & 0  \\
1 & 1&  \ldots & 1 & 2\Delta/p & \ldots & 0 \\
  t_1  & t_2  &  \ldots & t_d  & 0& \ldots & 0 \\
 \vdots & {} & { } & {} & \vdots & \ddots & \vdots\\
  t_1^{n} & t_2^{n} & \ldots & t_d^{n} &   0& \ldots & 2\Delta/p\\
 \end{array}\),
$$
where the first $d$ columns correspond to the evaluation of the polynomial
$\widetilde{f}$ in the points $t_1,\ldots,t_d$,
and the last $n+1$ columns to its coefficients, scaled such that all
coordinates of the wanted lattice point lie in an interval of length
$2\Delta+1$.
In the remainder of this section we focus on the first $d$ columns, that is, the
$d$-dimensional lattice $L$ of which the $d+n+1$~rows of the matrix
$$ \begin{pmatrix}
p & 0 & \cdots & 0 \\
0 & p & \cdots & 0 \\
\vdots & & \ddots & \vdots \\
0 & \cdots & 0 & p \\
1 & 1 & \cdots & 1 \\
t_1 & t_2 & \cdots & t_d \\
\vdots & \vdots & & \vdots \\
t_1^n & t_2^n & \cdots & t_d^n
\end{pmatrix} $$
are an overcomplete basis.

If $d\geq n+1$, a straightforward basis transformation that 
corresponds to Lagrange interpolation on $\Z_p$ eliminates $n+1$ rows and
transforms this basis into 
$$\begin{pmatrix}
p I_{d-n-1} & 0_{d\times(n+1)} \\
M_{(n+1)\times(d-n-1)} & I_{n+1}
\end{pmatrix},$$
from which it is clear that $\vol(L)=p^{d-n-1}$
(see Section~\ref{sec:Lat CVP} for the definition of $\vol(L)$).

We consider the case that $2^{n+1}\Delta < p$, so that, as shown in
Section~\ref{sec:Oscill Poly}
all functions $\widetilde{f}$ for which $\dmod{
\widetilde{f}(t)-f(t)}{p} \leq\Delta$ for all $t\in[-h,h]$
satisfy
$$ \widetilde{f}(t)-f(t)=\sum_{i=0}^n A_i \binom{t}{i}$$
for sufficiently small $A_0,\ldots,A_n$.
These differences correspond to the $n+1$-dimensional lattice $L_\text{approx}$
in $\Z^d$ of which the rows of
the matrix 
$$
B_\text{approx} = \begin{pmatrix}
1 & 1 & \cdots & 1 \\
t_1 & t_2 & \cdots & t_d \\
\begin{displaystyle} \binom{t_1}{2} \end{displaystyle}
 &\begin{displaystyle}  \binom{t_2}{2} \end{displaystyle}
 & \cdots & \begin{displaystyle} \binom{t_d}{2} \end{displaystyle}
 \\
\vdots & \vdots & & \vdots \\
\begin{displaystyle} \binom{t_1}{n} \end{displaystyle}
& \begin{displaystyle}  \binom{t_2}{n} \end{displaystyle}
 & \cdots & \begin{displaystyle} \binom{t_d}{n} \end{displaystyle}
\end{pmatrix}
$$ 
form a basis.
The volume of the lattice corresponding to these approximating polynomials is
given by
$$ \vol(L_\text{approx}) =  \sqrt{\det (B_\text{approx} B^T_\text{approx})}
$$ The determinant can be expanded as a sum of squares of
Vandermonde determinants, yielding
$$ \vol(L_\text{approx}) = 
\prod_{i=0}^n \frac{1}{i!} \times
\( \sum_{\substack{S\subseteq\{1,2,\ldots,d\} \\ \#S=n+1}}
\prod_{1\leq k < \ell \leq n+1} \( t_{S(k)} - t_{S(\ell)}\)^2 \)^{1/2},
$$ 
where the set $S=\{S(1),S(2),\ldots,S(n+1)\}$ runs over all
subsets of $\{1,2,\ldots,d\}$ with $n+1$ elements.

The lattice technique returns a polynomial $\widetilde{f}$ such that
the point
$$\left(\widetilde{f}(t_1)-f(t_1)
\ldots,\widetilde{f}(t_d)-f(t_d)\right)$$ lies in the lattice~$L$.
If this pointis also a lattice point of $L_\text{approx}$, then $\widetilde{f}$ approximates $f$ in all points of $[-h,h]$, and the attack succeeds.
So, from the attacker point of view, it is good if all lattice vectors of
$L$ that do not lie in $L_\text{approx}$ are long, that is, have the infinity
norm 
greater than $\Delta$, since then those points lie outside the hypercube around
the target vector.

We define the lattice $L^\perp$, obtained by orthogonally
projecting all points of $L$ onto
the hyperplane through the origin that is orthogonal to the basis vectors of
$L_\text{approx}$.
Obviously it holds that
$$ \vol(L^\perp) = \frac{\vol(L)
}{\vol(L_\text{approx})
}.$$
If $L^\perp$ does not have short vectors, its volume is large and the attack
succeeds. If its volume is small, $L^\perp$ has short vectors and, depending on
the number of independent short vectors and their lengths, the attack is likely
to fail.

This leads us to define the function
$$ S(n,h,p,\Delta,d) =
\log\( \frac{1}{d\Delta} \(
\frac{\vol(L)}{
\sqrt{\mathbb{E}_h [ \vol(L_\text{approx})^2 ] }}\)^{1/(d-n-1)}
 \).$$
It compares an esimate of the length of the shortest vector in
$L^\perp$, based on the Minkowski bound, to $\Delta$, in an average
sense. The cross-over between $L^\perp$ having short vectors or not
is at $S=0$.

The factor $1/d$ in the definition of the function $S$
is included to account for the orientations of $L_0$ and
of $L^\perp$ with respect to the coordinate axes, and the averaging is done in
such a way that it can be explicitly evaluated if
$d$ is much smaller that $h$ but is not too small:
\begin{align*}
\mathbb{E}_h &[ \vol(L_\text{approx})^2 ] =
\left( \prod_{i=1}^n \frac{1}{i!} \right)^2 \times \\
& \quad \frac{1}{(2h+1)^d}
\sum_{t_1,\ldots,t_d=-h}^h
\sum_{\substack{S\subseteq\{1,\ldots,d\} \\ \# S=n+1}}
\prod_{1\leq k < \ell \leq n+1}
(t_{S(k)} - t_{S(\ell)})^2 \\
&=
\left( \prod_{i=1}^n \frac{1}{i!} \right)^2 
\frac{\binom{d}{n+1}}{{(2h+1)}^{n+1}}
\sum_{t_1,\ldots,t_{n+1}=-h}^h
\prod_{1\leq k < \ell \leq n+1}
(t_k - t_\ell)^2 \\
&\approx
\left( \prod_{i=1}^n \frac{1}{i!} \right)^2 
\binom{d}{n+1} (2h)^{n(n+1)} \\
&\qquad  \idotsint\limits_{[0,1]^{n+1}} \prod_{1\leq
k<\ell\leq n+1} (x_k-x_\ell)^2 \,\mathrm{d}x_1\cdots\mathrm{d}x_{n+1}.
\end{align*}
The integral of the square of the Vandermonde determinant is equal to $(n+1)!$
times the determinant of the $(n+1)\times(n+1)$ Hilbert matrix
$H_{i,j}=1/(i+j-1)$, which has been calculated by Hilbert himself~\cite{Hilb}
and is equal to
$$
\det((H_{i,j})_{i,j=1}^{n+1}) =
\prod_{i=1}^n (i!)^4/\prod_{i=1}^{2n+1} i!.
$$
We thus obtain
\begin{equation}
\label{eq:formulaS}
\begin{split}
S(n,h,p&,\Delta,d) \approx \log\left(\frac{p}{d\Delta}\right) \\
&{}+\frac{1}{2(d-n-1)} \biggl( \sum_{i=1}^{n} (\log((n+1+i)!)-
\log(i!)) \\
&\qquad\qquad {} - \log\binom{d}{n+1} -n(n+1)\log(2h)  \biggr).
\end{split}
\end{equation}

\subsection{Tests}
We are mainly interested in parameter values corresponding to HIMMO, that is, when
we have a polynomial of degree $n$ and $b$-bit keys, we choose $p$ to have
precisely $(n+2)b$ bits and let  $\Delta=\lfloor p/2^{b+1}\rfloor$. We are
interested only in reconstruction in intervals of length at most $2^b$.
Note that in this case $S$ becomes independent of $p$.

Figure~\ref{fig:reconstruction} shows the reconstriction error
$({f}(t)-\widetilde{f}(t))/p$ versus $x/2^b$ for a randomly chosen polynomial
$f$ of degree $n=5$, $h=2^{15}$, 
$p$ a random~112-bit prime and $d=20$ and $23$,
respectively. The $d$ observation points are uniformly chosen in the interval
$[0,2^{16})$. For $d=20$, the error appears to be random for all $x$, whereas for $d=23$ the error is close to $0$ for small $x$.

\begin{figure}[H]
\centerline{%
\includegraphics[width=.45\textwidth]{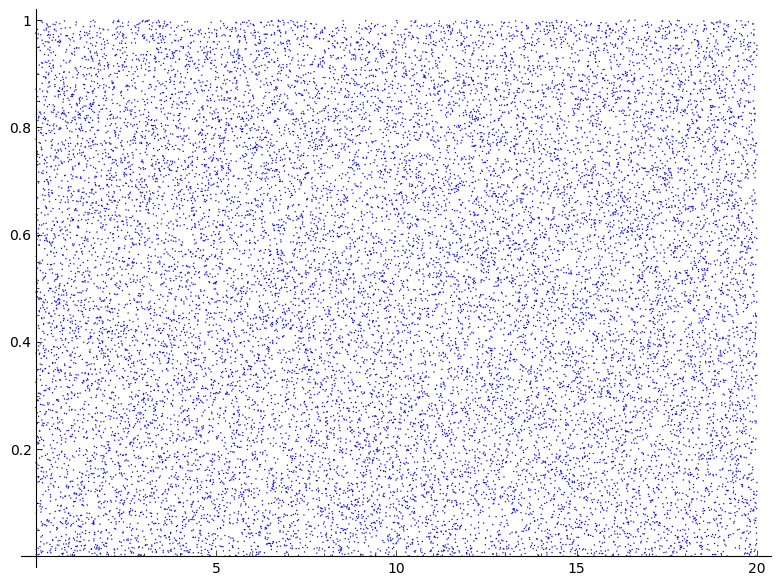}\hfil
\includegraphics[width=.45\textwidth]{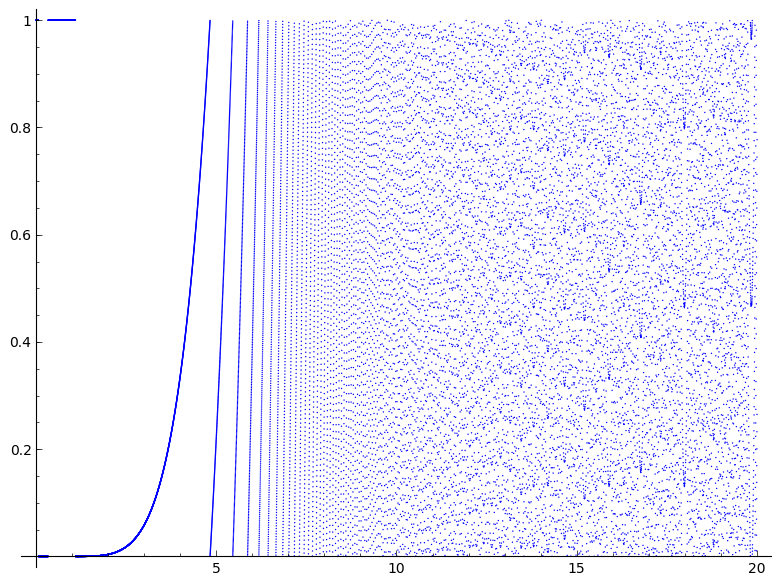}}
\caption{
Reconstruction error $({f}(t)-\widetilde{f}(t))/p$ versus
$x/2^b$ for $n=5$, $b=16$, $p$ a
random~112-bit prime and $d=20$ and $23$, respectively. The $d$ observation
points are uniformly chosen in the interval $[0,2^16)$.}
\label{fig:reconstruction}
\end{figure}

Figure~\ref{fig:graphS} shows graph of $S(5,2^{15},p,\lfloor p/2^{17}\rfloor,
d)$ as a function of $d$, showing that $S$ becomes positive at $d=23$.
This suggests that an approximate reconstruction is likely to be successful for
$d=23$, completely in line with our experimental results.
\begin{figure}[H]
\centerline{%
\includegraphics[width=.9\textwidth]{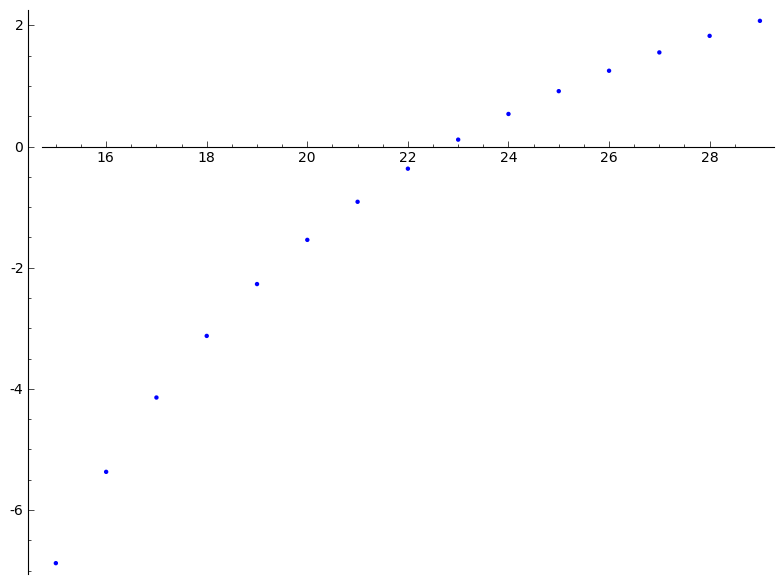}
}
\caption{%
The function $S(5,2^{15},p,\lfloor p/2^{17}\rfloor, d)$ as a function
of $d$ for large $p$.}
\label{fig:graphS}
\end{figure}

For larger degree $n=26$, and $h=2^{31}$, $p$ a $896$-bit prime, $\Delta=
\fl{p/2^{33}}$, the function $S$ becomes positive for $d=426$, which may be too large for
a successful lattice attack on the present computers. But in a shorter interval,
$h=128$, the function  $S$ becomes positive at $d=73$, see Figure~\ref{fig:S69-75}. 
\begin{figure}[H]
\centerline{%
\includegraphics[width=0.8\textwidth]{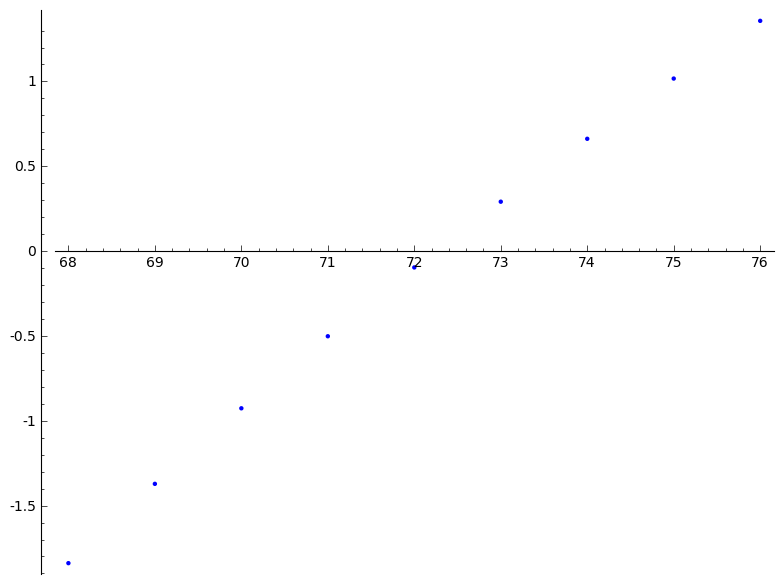} 
}
\caption{The function $S(26,2^{31},p,\fl{p/2^{33}},d)$ as a function
of $d$ for large $p$.}
\label{fig:S69-75}
\end{figure}

The experimental results in
Figures~\ref{fig:reconstructionsshort69}, ~\ref{fig:reconstructionsshort72}
and~\ref{fig:reconstructionsshort75} show how the reconstruction improves as
$d$ increases from $69$ to $75$.
For $c=69$, the reconstruction
is close to zero only in the observation points (note that the points 
near the horizontal line $y=1$ are also considered to be  close to zero). 
For $c=72$ also in a sizeable
fraction of the not-observed points. Finally for $c=75$ in nearly all points of an
interval of length $256$.

\begin{figure}[H]
\centerline{%
\includegraphics[width=0.8\textwidth]{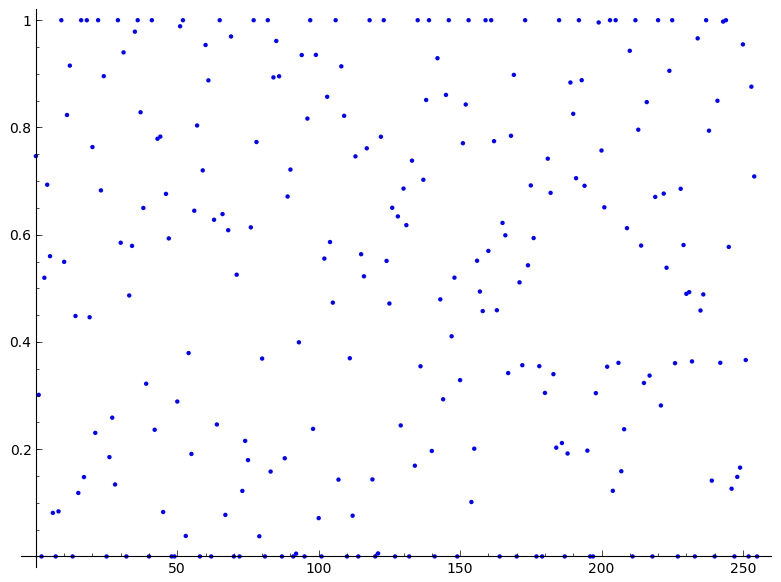} 
}
\caption{Reconstruction error for $0\leq x < 256$ with $\alpha=26$, $b=32$, 
$w=256$ for $d=69$.}
\label{fig:reconstructionsshort69}
\end{figure}

\begin{figure}[H]
\centerline{%
\includegraphics[width=0.8\textwidth]{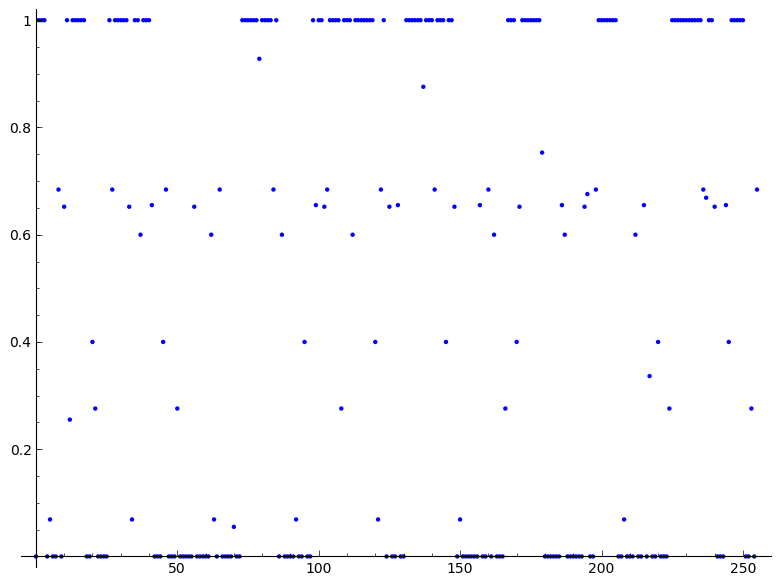}
}
\caption{Reconstruction error for $0\leq x < 256$ with $\alpha=26$, $b=32$, 
$w=256$ for $d=72$.
}
\label{fig:reconstructionsshort72}
\end{figure}

\begin{figure}[H]
\centerline{%
\includegraphics[width=0.8\textwidth]{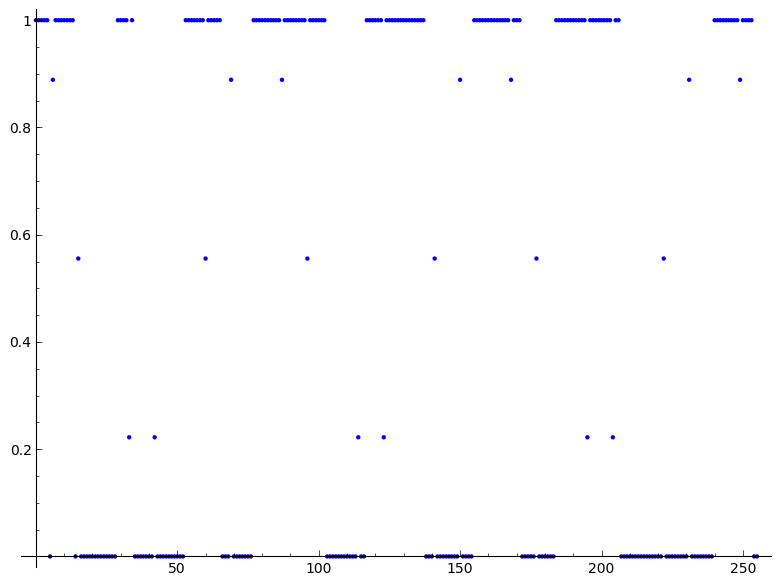}
}
\caption{Reconstruction error for $0\leq x < 256$ with $\alpha=26$, $b=32$, 
$w=256$ for $d=75$.}
\label{fig:reconstructionsshort75}
\end{figure}

Finally, we note that for fairly small values of $b$, for example, $b=8$, $S$ is
negative for all $d<2^b$ if $n$ is large enough, as shown for $n=10$ in
Figure~\ref{fig:plotSsmallb}.
This suggests that approximate reconstruction in the full interval $[0,255]$
cannot work, no matter how many observation points are used.

\begin{figure}[H]
\centerline{%
\includegraphics[width=.8\textwidth]{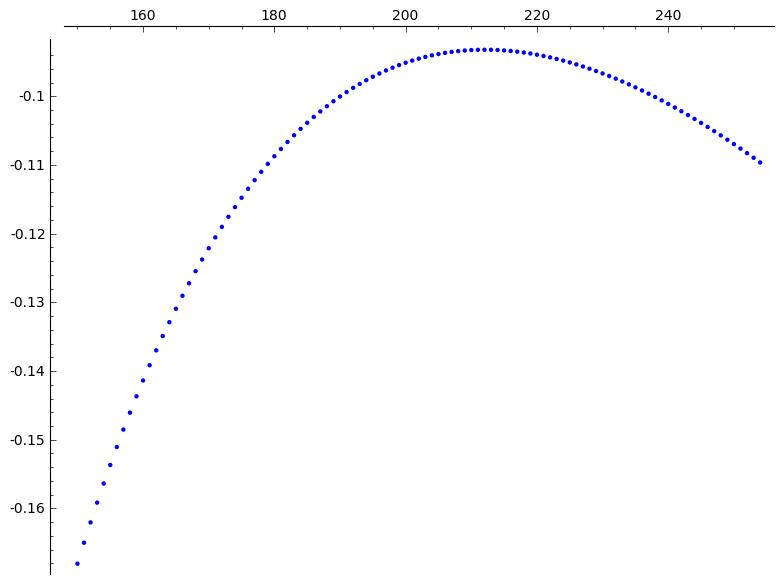}}
\caption{$S(10,2^7,p,\lfloor p/2^9\rfloor,d)$ as a function of $d$.}
\label{fig:plotSsmallb}
\end{figure}

\section{Comments}

Clearly in practical implementations of our algorithm, 
using a certain approximate version  of Lemma~\ref{lem:CVP}  is more natural.
One can easily check that the algorithm and the result of Theorem~\ref{thm:PolAppr} 
remain valid in this case as well.

In the settings of HIMMO with $b$-bit keys and $b$-bit identifiers
we have $h = 2^{b-1}$,  $2^{(n+2)b-1} < p < 2^{(n+2)b}$ and  
$\Delta=\lfloor p/2^{b+1} \rfloor$, see~\cite{GRTGGM},
so Theorem~\ref{thm:PolAppr} does not apply. This can be considered as
an indirect confirmation of the strength of HIMMO: recovery of the complete
polynomial is impossible.

From studying~\eqref{eq:formulaS} as a function of $d$ for various values of
$b$, we can obtain indications if an approximate recovery in an interval
$[-h,h]$ is likely to succeed, also for shorter intervals, $h<2^{b-1}$, that do
not contain all possible identifiers.

We also recall that there is a different approach to the hidden 
number problem, due to Akavia~\cite{Aka}. This approach does not 
use any lattice algorithms but rather examine the Fourier coefficients 
of the ``hidden'' linear function. However, the properties of these coefficients
(large values near the origin and a smooth decay away from the origin) 
do not hold for non-linear polynomial functions, where all Fourier coefficients
are expected to be of about the same size.


\begin{thebibliography}{99}
%

\bibitem{Aka} A. Akavia,
`Solving hidden number problem with one bit oracle and advice',
{\it  Lect. Notes in Comp. Sci.\/},
Springer-Verlag, Berlin, {\bf 5677} (2010), 337--354.

\bibitem{BV1} D. Boneh and R. Venkatesan,
`Hardness of computing the most significant bits of secret keys in
Diffie--Hellman and related schemes', {\it Lect. Notes in Comp.
Sci.\/}, Springer-Verlag, Berlin, {\bf 1109} (1996), 129--142.

\bibitem{BV2} D. Boneh and R. Venkatesan,
`Rounding in lattices and its cryptographic applications', {\it
Proc. 8th Annual ACM-SIAM Symp. on Discr. Algorithms\/}, SIAM, 1997,  675--681.
%


\bibitem{CCGHSZ} M.-C. Chang, J. Cilleruelo, M. Z. Garaev, J.  Hern\'andez,
I. E. Shparlinski and A. Zumalac\'{a}rregui,
`Points on curves in small boxes and applications',
{\it Preprint\/}, 2011 (available from {\tt http://arxiv.org/abs/1111.1543}).


\bibitem{CGOS} J. Cilleruelo, M. Z. Garaev,  A. Ostafe and
I. E. Shparlinski,
`On the concentration of points of polynomial maps
and applications',
{\it Math. Zeit.\/}, {\bf 272} (2012), 825--837.

\bibitem{CPR} T. Cochrane,  C. Pinner and J. Rosenhouse,
`Bounds on exponential sums and the polynomial Waring's problem $\bmod\, p$',
{\it Proc. Lond. Math. Soc.\/}, {\bf 67} (2003), 319--336.

\bibitem{ConSlo}
J.~H.~Conway and N.~J.~A.~Sloane,
{\it Sphere packings, lattices and groups\/}, 3rd edition.
Grundlehren der Mathematischen Wissenschaften, vol.~290, 
Springer-Verlag, New York, 1999.



\bibitem{GRTGGM} O. Garcia-Morchon, R. Rietman,  L. Tolhuizen,
D. G\'omez-P\'erez, J. Gutierrez and S. Merino del Pozo,
`An ultra-lightweight ID-based pairwise key establishment 
scheme aiming at full collusion resistance',
{\it ePrint Archive\/}, Report~618, 2012 (available 
from {\tt http://eprint.iacr.org/2012/618}).

\bibitem{vzGShp}
J.~von~zur Gathen and I. E. Shparlinski,
`Polynomial interpolation from multiples',  {\it Proc. 15th 
ACM-SIAM Symposium on Discrete Algorithms\/}, SIAM, 2004, 1125--1130.

\bibitem{GrLoSch} M. Gr{\"o}tschel, L. Lov{\'a}sz and A. Schrijver,
{\it Geometric algorithms and combinatorial optimization\/},
 Algorithms and Combinatorics: Study and Research Texts, vol.~2,
Springer-Verlag, Berlin, 1993.

\bibitem{GrLe}  P. M. Gruber and C. G.  Lekkerkerker,
`Geometry of numbers', 2nd edition, North-Holland Mathematical Library, 
vol.~37. North-Holland Publishing Co., Amsterdam, 1987.

\bibitem{Hilb}
D. Hilbert, `Ein Beitrag zur Theorie des Legendre'schen Polynoms',
{\it Acta Mathematica\/}, {\bf 18} (1894), 155–-159.

\bibitem{Iv}
A. Ivi{\'c}, 
`On a problem of Erd{\H o}s involving the largest 
prime factor of $n$',
{\it Monatsh. Math.\/}, {\bf 145} (2005), 35--46. 


\bibitem{IwKow} H. Iwaniec and E. Kowalski,
{\it Analytic number theory\/}, Amer.  Math.  Soc.,
Providence, RI, 2004.



\bibitem{Ker} B.  Kerr, `Solutions to polynomial 
congruences in well shaped sets',
 {\it Bull. Aust. Math. Soc.\/},  {\bf 88} (2013), 435--447.

\bibitem{LLL} A. K. Lenstra, H. W. Lenstra and L.  Lov{\'a}sz,
`Factoring polynomials with rational coefficients', {\it
Mathem. Ann.\/}, {\bf 261} (1982), 515--534.

\bibitem{LN} R. Lidl and H. Niederreiter, {\it Finite fields\/},
Cambridge Univ. Press, Cambridge, 1997.

\bibitem{Micc}  D. Micciancio,
{\it On the hardness of the shortest vector problem\/}, PhD Thesis, MIT, 1998.

\bibitem{MicVou}
D. Micciancio and P. Voulgaris, `A deterministic single exponential time algorithm
for most lattice problems based on Voronoi cell computations',
{\it SIAM J. Comp.\/}, {\bf 42} (2013), 1364--1391.

\bibitem{Mont} H. L. Montgomery, {\it Ten lectures on the interface
between analytic number theory and harmonic analysis},  Amer. Math.
Soc.,  Providence, RI, 1994.

\bibitem{Ngu}
P. Q.~Nguyen, `Public-key cryptanalysis', 
{\it Recent Trends in Cryptography\/}, 
Contemp. Math., vol.477, Amer. 
Math. Soc., 2009,  67--119.

\bibitem{NgSt1} P. Q.~Nguyen and J. Stern, 
`Lattice reduction in cryptology:  An update', {\it Lect. Notes in Comp.
Sci.\/}, Springer-Verlag, Berlin, {\bf 1838} (2000), 85--112.

\bibitem{NgSt2}
P. Q.~Nguyen  and J. Stern, `The two faces of lattices in
cryptology', {\it Lect. Notes in Comp.  Sci.\/}, Springer-Verlag,
Berlin, {\bf 2146} (2001), 146--180.

\bibitem{Reg} O. Regev, `On the complexity of lattice problems with 
polynomial approximation factors', {\it The LLL Algorithm:Surveys and 
Applications\/},  Springer-Verlag, 2010, 475--496.


\bibitem{Shp1}
I.~E.~Shparlinski, `Sparse polynomial approximation in finite
fields', {\it Proc. 33rd ACM Symp. on Theory of Comput.\/}, Crete,
Greece, July 6-8, (2001), 209--215.

\bibitem{Shp2}  
I.~E.~Shparlinski,  `Playing ``Hide-and-Seek''   with numbers:  The
hidden number problem, lattices and exponential sums',
{\it Proc. Symp. in Appl. Math.\/}, Amer. Math. Soc., Providence, RI,
{\bf 62} (2005), 153--177.

 \bibitem{ShpWint}
I.~E.~Shparlinski and A. Winterhof,  
`Noisy interpolation  of sparse polynomials in finite fields',
{\it Appl. Algebra in Engin., Commun. and Computing\/},   
{\bf 16} (2005), 307--317. 

\bibitem{Wool1} T.~D.~Wooley,
`Vinogradov's mean value theorem via efficient congruencing',
{\it Ann. Math.\/}, {\bf 175} (2012), 1575--1627.


\bibitem{Wool2} T.~D.~Wooley,
`Vinogradov's mean value theorem via efficient congruencing, II',
{\it Duke Math. J.\/}, {\bf 162} (2013), 673--730.



\bibitem{Wool3} T.~D.~Wooley,
`Multigrade efficient congruencing and Vinogradov's mean value theorem',
{\it Preprint\/}, 2011 (available from {\tt http://arxiv.org/abs/1310.8447}).

\end{thebibliography}
\end{document}